\documentclass{article}
\usepackage{agd-ktrees}

\usepackage[margin=1.0in]{geometry}

\author{Andrew Gainer-Dewar\\
  \small Department of Mathematics\\
  \small Carleton College\\
  \small Northfield, MN, U.S.A.\\
  \small\tt againerdewar@carleton.edu
}

\title{$\Gamma$-species and the enumeration of $k$-trees}

\begin{document}
\maketitle
\begin{abstract}
  We study the class of graphs known as $k$-trees through the lens of Joyal's theory of combinatorial species (and an extension known as `$\Gamma$-species' which incorporates data about `structural' group actions).
  This culminates in a system of recursive functional equations giving the generating function for unlabeled $k$-trees which allows for fast, efficient computation of their numbers.
  Enumerations up to $k = 10$ and $n = 30$ (for a $k$-tree with $n + k - 1$ vertices) are included in tables, and Sage code for the general computation is included in an appendix.
\end{abstract}

\section{Introduction}\label{s:intro}
\subsection{$k$-trees}\label{ss:ktrees}
Trees and their generalizations have played an important role in the literature of combinatorial graph theory throughout its history.
The multi-dimensional generalization to so-called `$k$-trees' has proved to be particularly fertile ground for both research problems and applications.

The class $\kt{k}$ of $k$-trees (for $k \in \ringname{N}$) may be defined recursively:
\begin{definition}
  \label{def:ktree}
  The complete graph on $k$ vertices ($K_{k}$) is a $k$-tree, and any graph formed by adding a single vertex to a $k$-tree and connecting that vertex by edges to some existing $k$-clique (that is, induced $k$-complete subgraph) of that $k$-tree is a $k$-tree.
\end{definition}

The graph-theoretic notion of $k$-trees was first introduced in 1968 in \cite{harpalm:acycsimp}; vertex-labeled $k$-trees were quickly enumerated in the following year in both \cite{moon:lktrees} and \cite{beinpipp:lktrees}.
The special case $k=2$ has been especially thoroughly studied; enumerations are available in the literature for edge- and triangle-labeled $2$-trees in \cite{palm:l2trees}, for plane $2$-trees in \cite{palmread:p2trees}, and for unlabeled $2$-trees in \cite{harpalm:acycsimp} and \cite{harpalm:graphenum}.
In 2001, the theory of species was brought to bear on $2$-trees in \cite{gessel:spec2trees}, resulting in more explicit formulas for the enumeration of unlabeled $2$-trees.
An extensive literature on other properties of $k$-trees and their applications has also emerged; Beineke and Pippert claim in \cite{beinpipp:multidim} that ``[t]here are now over 100 papers on various aspects of $k$-trees''.
However, no general enumeration of unlabeled $k$-trees appears in the literature to date.

Although we do not derive a closed form for the number of $k$-trees, the work in this paper does permit efficient recursive computation of their generating function.
A formula for this generating function is given in \cref{cor:ktgf} using components defined recursively in \cref{cor:ctogf}.

To begin, we establish two definitions for substructures of $k$-trees which we will use extensively in our analysis.
\begin{definition}
  \label{def:hedfront}
  A \emph{hedron} of a $k$-tree is a $\pbrac{k+1}$-clique and a \emph{front} is a $k$-clique.
\end{definition}
We will frequently describe $k$-trees as assemblages of hedra attached along their fronts rather than using explicit graph-theoretic descriptions in terms of edges and vertices, keeping in mind that the structure of interest is graph-theoretic and not geometric.
The recursive addition of a single vertex and its connection by edges to an existing $k$-clique in \cref{def:ktree} is then interpreted as the attachment of a hedron to an existing one along some front, identifying the $k$ vertices they have in common.
The analogy to the recursive definition of conventional trees is clear, and in fact the class $\mathfrak{a}$ of trees may be recovered by setting $k = 1$.
For higher $k$, the structures formed are still distinctively tree-like; for example, $2$-trees are formed by gluing triangles together along their edges without forming loops of triangles (see \cref{fig:ex2tree}), while $3$-trees are formed by gluing tetrahedra together along their triangular faces without forming loops of tetrahedra.

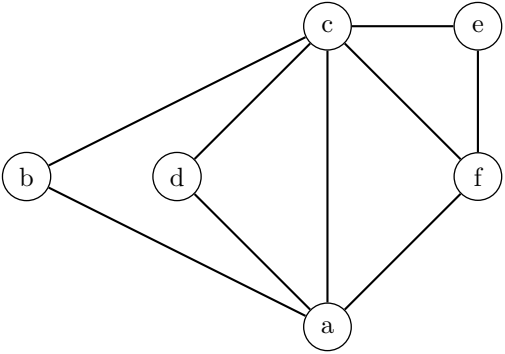
\begin{figure}[htb]
  \centering
  \begin{tikzpicture}
    \SetGraphUnit{2}
    \GraphInit[vstyle=normal]
    \Vertex{a}

    \NOWE(a){d}
    \Edge(a)(d)

    \NOEA(d){c}
    \Edge(a)(c)
    \Edge(d)(c)
    
    \WE(d){b}
    \Edge(a)(b)
    \Edge(c)(b)
    
    \NOEA(a){f}
    \Edge(a)(f)
    \Edge(c)(f)
    
    \NO(f){e}
    \Edge(f)(e)
    \Edge(c)(e)
  \end{tikzpicture}
  \caption{A (vertex-labeled) $2$-tree}
  \label{fig:ex2tree}
\end{figure}

In graph-theoretic contexts, it is conventional to label graphs on their vertices and possibly their edges.
However, for our purposes, it will be more convenient to label hedra and fronts.
Throughout, we will treat the species $\kt{k}$ of $k$-trees as a two-sort species, with $X$-labels on the hedra and $Y$-labels on their fronts; in diagrams, we will generally use capital letters for the hedron-labels and positive integers for the front-labels (see \cref{fig:exlab2tree}).
A formula for the cycle index of the species $\kt{k}$ is given in \cref{thm:ktci} using components defined recursively in \cref{thm:ctyfuncci,thm:ctxyfuncci}.
(Readers unfamiliar with the theory of species and its applications to graph enumeration may find a full exposition of the subject in \cite{bll:species}, which also will serve as a reference for any unexplained notation.)

\begin{figure}[htb]
  \centering
  \begin{tikzpicture}
    \SetGraphUnit{3}
    \GraphInit[vstyle=Hasse]
    \SetUpEdge[labelstyle={draw}]

    \Vertex{a}

    \NOWE(a){d}
    \Edge[label=2](a)(d)

    \NOEA(d){c}
    \Edge[label=4](a)(c)
    \Edge[label=5](d)(c)

    \node at (barycentric cs:a=1,d=1,c=1) {B};
    
    \WE(d){b}
    \Edge[label=1](a)(b)
    \Edge[label=6](c)(b)

    \node at (barycentric cs:b=1,d=1) {D};
    
    \NOEA(a){f}
    \Edge[label=9](a)(f)
    \Edge[label=3](c)(f)

    \node at (barycentric cs:a=1,c=1,f=1) {C};
    
    \NO(f){e}
    \Edge[label=7](f)(e)
    \Edge[label=8](c)(e)

    \node at (barycentric cs:f=1,c=1,e=1) {A};

  \end{tikzpicture}
  \caption{A (hedron-and-front--labeled) $2$-tree}
  \label{fig:exlab2tree}
\end{figure}
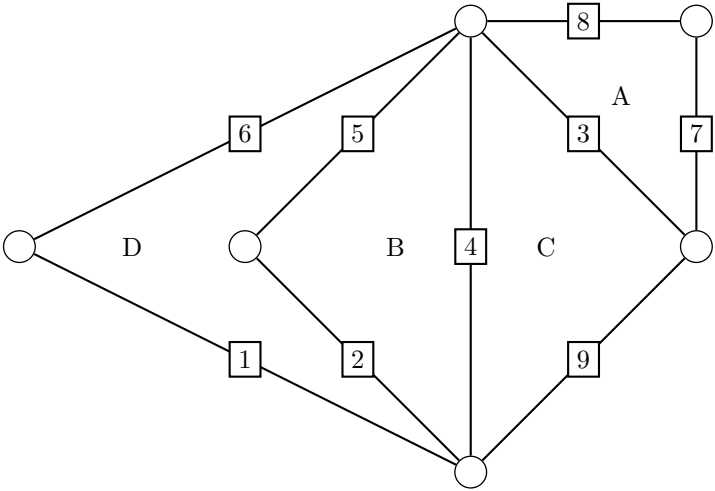

\section{The dissymmetry theorem for $k$-trees}\label{s:dissymk}
Studies of tree-like structures---especially those explicitly informed by the theory of species, as ours will be---often feature decompositions based on \emph{dissymmetry}, which allow enumerations of unrooted structures to be recharacterized in terms of rooted structures.
For example, as seen in \cite[\S 4.1]{bll:species}, the species $\mathfrak{a}$ of trees and $\specname{A} = \pointed{\mathfrak{a}}$ of rooted trees are related by the equation
\begin{equation*}
  \specname{A} + \specname{E}_{2} \pbrac{\specname{A}} = \mathfrak{a} + \specname{A}^{2}
\end{equation*}
where the proof hinges on a recursive structural decomposition of trees.
In this case, the species $\specname{A}$ is relatively easy to characterize explicitly, so this equation serves to characterize the species $\mathfrak{a}$, which would be difficult to do directly.

A similar theorem holds for $k$-trees.
\begin{theorem}
  \label{thm:dissymk}
  The species $\ktx{k}$ and $\kty{k}$ of $k$-trees rooted at hedra and fronts respectively, $\ktxy{k}$ of $k$-trees rooted at a hedron with a designated front, and $\kt{k}$ of unrooted $k$-trees are related by the equation
  \begin{equation}
    \label{eq:dissymk}
    \ktx{k} + \kty{k} = \kt{k} + \ktxy{k}
  \end{equation}
  as an isomorphism of species.
\end{theorem}

\begin{proof}
  We give a bijective, natural (\foreign{i.e.}~label-equivariant) map from $\pbrac{\ktx{k} + \kty{k}}$-structures on the left side to $\pbrac{\kt{k} + \ktxy{k}}$-structures on the right side.
  Define a \emph{$k$-path} in a $k$-tree to be a non-self-intersecting sequence of consecutively adjacent hedra and fronts, and define the \emph{length} of a $k$-path to be the total number of hedra and fronts along it.
  Note that the ends of every maximal $k$-path in a $k$-tree are fronts.
  It is easily verified, as in \cite{kob:ktlogspace}, that every $k$-tree has a unique \emph{center} clique (either a hedron or a front) which is the midpoint of every longest $k$-path (or, equivalently, has the greatest $k$-eccentricity, defined appropriately).
  
  An $\pbrac{\ktx{k} + \kty{k}}$-structure on the left-hand side of the equation is a $k$-tree $T$ rooted at some clique $c$, which is either a hedron or a front.
  Suppose that $c$ is the center of $T$.
  We then map $T$ to its unrooted equivalent in $\kt{k}$ on the right-hand side.
  This map is a natural bijection from its preimage, the set of $k$-trees rooted at their centers, to $\kt{k}$, the set of unrooted $k$-trees.

  Now suppose that the root clique $c$ of the $k$-tree $T$ is \emph{not} the center, which we denote $C$.
  Identify the clique $c'$ which is adjacent to $c$ along the $k$-path from $c$ to $C$.
  We then map the $k$-tree $T$ rooted at the clique $c$ to the same tree $T$ rooted at \emph{both} $c$ and its neighbor $c'$.
  This map is also a natural bijection, in this case from the set of $k$-trees rooted at vertices which are \emph{not} their centers to the set $\ktxy{k}$ of $k$-trees rooted at an adjacent hedron-front pair.

  Since these maps are label-equivariant bijections, they induce an isomorphism of species
  \begin{equation*}
    \ktx{k} + \kty{k} = \kt{k} + \ktxy{k}
  \end{equation*}
  as desired, completing the proof.
\end{proof}

In general we will reformulate the dissymmetry theorem as follows:
\begin{corollary}
  \label{cor:dissymkreform}
  For the various forms of the species $\kt{k}$ as above, we have
  \begin{equation}
    \label{eq:dissymkreform}
    \kt{k} = \ktx{k} + \kty{k} - \ktxy{k}.
  \end{equation}
  as an isomorphism of species.
\end{corollary}

This species subtraction is well-defined in the sense that since the species $\ktxy{k}$ embeds in the species $\ktx{k} + \kty{k}$ by the centering map described in the proof of \cref{thm:dissymk}.
Essentially, \cref{eq:dissymkreform} identifies each unrooted $k$-tree with itself rooted at its center simplex.
This may be understood formally either in the sense of virtual species as in \cite[\S 2.5]{bll:species} or in the sense of species maps as in \cite[Def.~1.3.3]{agd:thesis}.

\Cref{thm:dissymk} and the consequent \cref{eq:dissymkreform} allow us to reframe enumerative questions about generic $k$-trees in terms of questions about $k$-trees rooted in various ways.
However, the rich internal symmetries of large cliques obstruct direct analysis of these rooted structures.
We need to break these symmetries to proceed.

\section{Coherently-oriented $k$-trees}
\subsection{Symmetry-breaking}\label{ss:symbreak}
In the case of the species $\specname{A} = \pointed{\kt{1}}$ of rooted trees, we may obtain a simple recursive functional equation \cite[\S 1, eq.~(9)]{bll:species}:
\begin{equation}
  \label{eq:rtrees}
  \specname{A} = X \cdot \specname{E} \pbrac{\specname{A}}.
\end{equation}
This completely characterizes the combinatorial structure of the class of trees.

However, in the more general case of $k$-trees, no such simple relationship obtains; attached to a given hedron is a collection of sets of hedra (one such set per front), but simply specifying which fronts to attach to which does not fully specify the attachings, and the structure of that collection of sets is complex.
We will break this symmetry by adding additional structure which we can later remove using the theory of quotient species.

\begin{definition}
  \label{def:mirrorfronts}
  Let $h_{1}$ and $h_{2}$ be two hedra joined at a front $f$, hereafter said to be \emph{adjacent}.
  Each other front of one of the hedra shares $k-1$ vertices with $f$; we say that two fronts $f_{1}$ of $h_{1}$ and $f_{2}$ of $h_{2}$ are \emph{mirror with respect to $f$} if these shared vertices are the same, or equivalently if $f_{1} \cap f = f_{2} \cap f$.
\end{definition}

\begin{observation}
  \label{obs:mirrorfronts}
  Let $T$ be a $k$-tree with two hedra $h_{1}$ and $h_{2}$ joined at a front $f$.
  Then there is exactly one front of $h_{2}$ mirror to each front of $h_{1}$ with respect to their shared front $f$.
\end{observation}

\begin{definition}
  \label{def:coktree}
  Define an \emph{orientation} of a hedron to be a cyclic ordering of the set of its fronts and an \emph{orientation} of a $k$-tree to be a choice of orientation for each of its hedra.
  If two oriented hedra share a front, their orientations are \emph{compatible} if they correspond under the mirror bijection.
  Then an orientation of a $k$-tree is \emph{coherent} if every pair of adjacent hedra is compatibly-oriented.
\end{definition}
See \cref{fig:exco2tree} for an example.
Note that every $k$-tree admits many coherent orientations---any one hedron of the $k$-tree may be oriented freely, and a unique orientation of the whole $k$-tree will result from each choice of such an orientation of one hedron.
We will denote by $\ktco{k}$ the species of coherently-oriented $k$-trees.

By shifting from the general $k$-tree setting to that of coherently-oriented $k$-trees, we break the symmetry described above.
If we can now establish a group action on $\ktco{k}$ whose orbits are generic $k$-trees we can use the theory of quotient species to extract the generic species $\kt{k}$.
First, however, we describe an encoding procedure which will make future work more convenient.

\begin{figure}[htb]
  \centering
  \begin{tikzpicture}
    \SetGraphUnit{3}
    \GraphInit[vstyle=Hasse]
    \SetUpEdge[labelstyle={draw}]

    \Vertex{a}

    \NOWE(a){d}
    \Edge[label=2](a)(d)

    \NOEA(d){c}
    \Edge[label=4](a)(c)
    \Edge[label=5](d)(c)

    \coordinate (B) at (barycentric cs:a=1,d=1.25,c=1);
    \node at (B) {B};
    \cyccwr[.5cm]{(B)};
    
    \WE(d){b}
    \Edge[label=1](a)(b)
    \Edge[label=6](c)(b)

    \coordinate (D) at (barycentric cs:b=1,d=1.25);
    \node at (D) {D};
    \cyccwr[.5cm]{(D)};
    
    \NOEA(a){f}
    \Edge[label=9](a)(f)
    \Edge[label=3](c)(f)

    \coordinate (C) at (barycentric cs:a=1,c=1,f=1.25);
    \node at (C) {C};
    \cycccwl[.5cm]{(C)};
    
    \NO(f){e}
    \Edge[label=7](f)(e)
    \Edge[label=8](c)(e)

    \coordinate (A) at (barycentric cs:f=1,c=1,e=1.75);
    \node at (A) {A};
    \cyccwl[.5cm]{(A)};
    
  \end{tikzpicture}
  \caption{A coherently-oriented $2$-tree}
  \label{fig:exco2tree}
\end{figure}
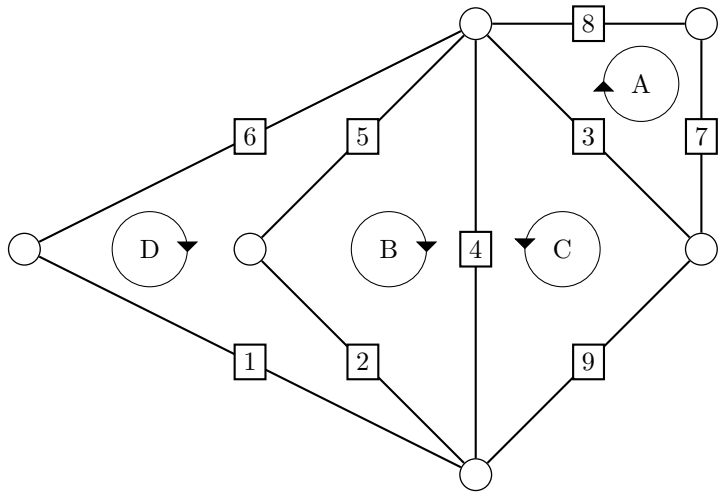

\subsection{Bicolored tree encoding}\label{ss:bctree}
Although $k$-trees are graphs (and hence made up simply of edges and vertices), their structure is more conveniently described in terms of their simplicial structure of hedra and fronts.
Indeed, if each hedron has an orientation of its faces and we choose in advance which hedra to attach to which by what fronts, the requirement that the resulting $k$-tree be coherently oriented is strong enough to characterize the attaching completely.
We thus pass from coherently-oriented $k$-trees to a surrogate structure which exposes the salient features of this attaching structure more clearly---structured bicolored trees in the spirit of the $R, S$-enriched bicolored trees of \cite[\S 3.2]{bll:species}.

A $\pbrac{\specname{C}_{k+1}, \specname{E}}$-enriched bicolored tree is a bicolored tree each black vertex of which carries a $\specname{C}_{k+1}$-structure (that is, a cyclic ordering on $k+1$ elements) on its white neighbors.
(The $\specname{E}$-structure on the black neighbors of each white vertex is already implicit in the bicolored tree itself.)
For later convenience, we will sometimes call such objects \emph{$k$-coding trees}, and we will denote by $\ct{k}$ the species of such $k$-coding trees.

We now define a map $\beta: \ktco{k} \sbrac{n} \to \ct{k} \sbrac{n}$.
For a given coherently-oriented $k$-tree $T$ with $n$ hedra:
\begin{itemize}
  \item For every hedron of $T$ construct a black vertex and for every front a white vertex, assigning labels appropriately.
  \item For every black-white vertex pair, construct a connecting edge if the white vertex represents a front of the hedron represented by the black vertex.
  \item Finally, enrich the collection of neighbors of each black vertex with a $\specname{C}_{k+1}$-structure inherited directly from the orientation of the $k$-tree $T$.
\end{itemize}
The resulting object $\beta \pbrac{T}$ is clearly a $k$-coding tree with $n$ black vertices.

We can recover a $T$ from $\beta \pbrac{T}$ by following the reverse procedure.
For an example, see \cref{fig:exbctree}, which shows the $2$-coding tree associated to the coherently-oriented $2$-tree of \cref{fig:exco2tree}.
Note that, for clarity, we have rendered the black vertices (corresponding to hedra) with squares.

\begin{figure}[htb]
  \centering
  \begin{tikzpicture}
    \SetGraphUnit{1.5}
    \GraphInit[vstyle=normal]

    \Vertex[style=ynode]{4}

    \SOWE[style=xnode](4){B}
    \cyccwr{(B)};
    \Edge(4)(B)

    \WE[style=ynode](B){5}
    \Edge(B)(5)

    \SOEA[style=ynode](B){2}
    \Edge(B)(2)

    \NOWE[style=xnode](4){D}
    \cyccwr{(D)};
    \Edge(4)(D)
    
    \WE[style=ynode](D){1}
    \Edge(D)(1)

    \NOEA[style=ynode](D){6}
    \Edge(D)(6)

    \EA[style=xnode](4){C}
    \cycccwr{(C)};
    \Edge(4)(C)

    \SOEA[style=ynode](C){9}
    \Edge(C)(9)

    \NOEA[style=ynode](C){3}
    \Edge(C)(3)

    \NO[style=xnode](3){A}
    \cyccwl{(A)};
    \Edge(3)(A)

    \NOWE[style=ynode](A){8}
    \Edge(A)(8)

    \NOEA[style=ynode](A){7}
    \Edge(A)(7)
    
  \end{tikzpicture}
  \caption{A $\pbrac{\specname{C}_{k+1}, \specname{E}}$-enriched bicolored tree encoding a coherently-oriented $2$-tree}
  \label{fig:exbctree}
\end{figure}
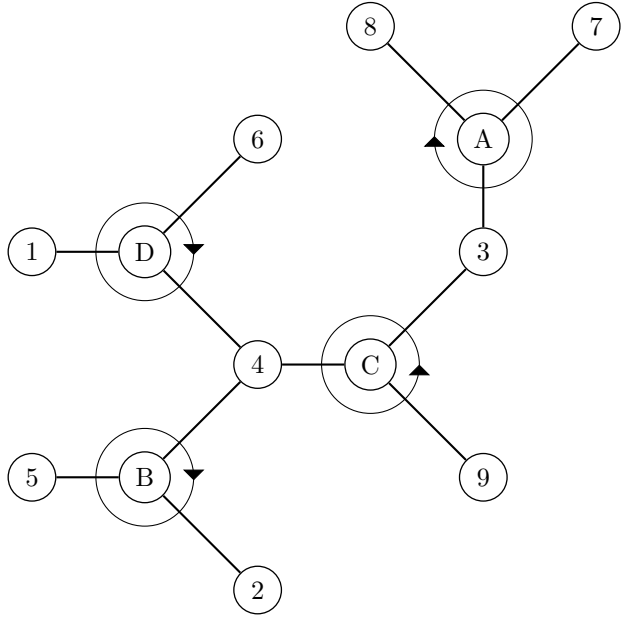

\begin{theorem}\label{thm:bctreeenc}
  The map $\beta$ induces an isomorphism of species $\ktco{k} \simeq \ct{k}$.
\end{theorem}

\begin{proof}
  It is clear that $\beta$ sends each coherently-oriented $k$-tree to a unique $k$-coding tree, and that this map commutes with permutations on the label sets (and thus is categorically natural).
  To show that $\beta$ induces a species isomorphism, then, we need only show that $\beta$ is a surjection onto $\ct{k} \sbrac{n}$ for each $n$.
  Throughout, we will say `$F$ and $G$ have contact of order $n$' when the restrictions $F_{\leq n}$ and $G_{\leq n}$ of the species $F$ and $G$ to label sets of cardinality at most $n$ are naturally isomorphic.
  
  First, we note that there are exactly $k!$ coherently-oriented $k$-trees with one hedron---one for each cyclic ordering of the $k+1$ front labels.
  There are also $k!$ coding trees with one black vertex, and the encoding $\beta$ is clearly a natural bijection between these two sets.
  Thus, the species $\ktco{k}$ of coherently-oriented $k$-trees and $\ct{k}$ of $k$-coding trees have contact of order $1$. 

  Now, by way of induction, suppose $\ktco{k}$ and $\ct{k}$ have contact of order $n \geq 1$.
  Let $C$ be a $k$-coding tree with $n+1$ black vertices.
  Then let $C_{1}$ and $C_{2}$ be two distinct sub-$k$-coding trees of $C$, each obtained from $C$ by removing one black node which has only one white neighbor which is not a leaf.
  Then, by hypothesis, there exist coherently-oriented $k$-trees $T_{1}$ and $T_{2}$ with $n$ hedra such that $\beta \pbrac{T_{1}} = C_{1}$ and $\beta \pbrac{T_{2}} = C_{2}$.
  Moreover, $\beta \pbrac{T_{1} \cap T_{2}} = \beta \pbrac{T_{1}} \cap \beta \pbrac{T_{2}}$, and this $k$-coding tree has $n-1$ black vertices, so $T_{1} \cap T_{2}$ has $n-1$ hedra.
  Thus, $T = T_{1} \cup T_{2}$ is a coherently-oriented $k$-tree with $n+1$ black hedra, and $\beta \pbrac{T} = C$ as desired.
  Thus, $\beta^{-1} \pbrac{\beta \pbrac{T_{1}} \cup \beta \pbrac{T_{2}}} = T_{1} \cup T_{2} = T$, and hence $\ktco{k}$ and $\ct{k}$ have contact of order $n+1$.
\end{proof}

Thus, $\ktco{k}$ and $\ct{k}$ are isomorphic as species; however, $k$-coding trees are much simpler than coherently-oriented $k$-trees as graphs.
Moreover, $k$-coding trees are doubly-enriched bicolored trees as in \cite[\S 3.2]{bll:species}, for which the authors of that text develop a system of functional equations which fully characterizes the cycle index of such a species.
We thus will proceed in the following sections with a study of the species $\ct{k}$, then lift our results to the $k$-tree context.

\subsection{Functional decomposition of $k$-coding trees}\label{ss:codecomp}
With the encoding $\beta: \ktco{k} \to \ct{k}$, we now have direct graph-theoretic access to the attaching structure of coherently-oriented $k$-trees.
We therefore turn our attention to the $k$-coding trees themselves to produce a recursive decomposition.
As with $k$-trees, we will study rooted versions of the species $\ct{k}$ of $k$-coding trees first, then use dissymmetry to apply the results to unrooted enumeration.

\begin{theorem}
  \label{thm:funcdecompct}
  The species $\ctx{k}$ of $X$-rooted $k$-coding trees, $\cty{k}$ of $Y$-rooted $k$-coding trees, and $\ctxy{k}$ of edge-rooted $k$-coding trees satisfy the functional equations
  \begin{subequations}
    \label{eq:ctfunc}
    \begin{align}
      \ctx{k} &= X \cdot \specname{C}_{k+1} \pbrac[\big]{\cty{k}} \label{eq:ctxfunc} \\
      \cty{k} &= Y \cdot \specname{E} \pbrac*{X \cdot \specname{L}_{k} \pbrac[\big]{\cty{k}}} \label{eq:ctyfunc} \\
      \ctxy{k} &= \cty{k} \cdot X \cdot \specname{L}_{k} \pbrac[\big]{\cty{k}} = X \cdot \specname{L}_{k+1} \pbrac[\big]{\cty{k}} \label{eq:ctxyfunc}
    \end{align}
  \end{subequations}
  as isomorphisms of species.
\end{theorem}

\begin{proof}
  By construction, a $\ctx{k}$-structure consists of a single $X$-label and a cyclically-ordered $\pbrac{k+1}$-set of $\cty{k}$-structures.
  This gives \cref{eq:ctxfunc}.
  See \cref{fig:ctxconst} for an example of this construction.

  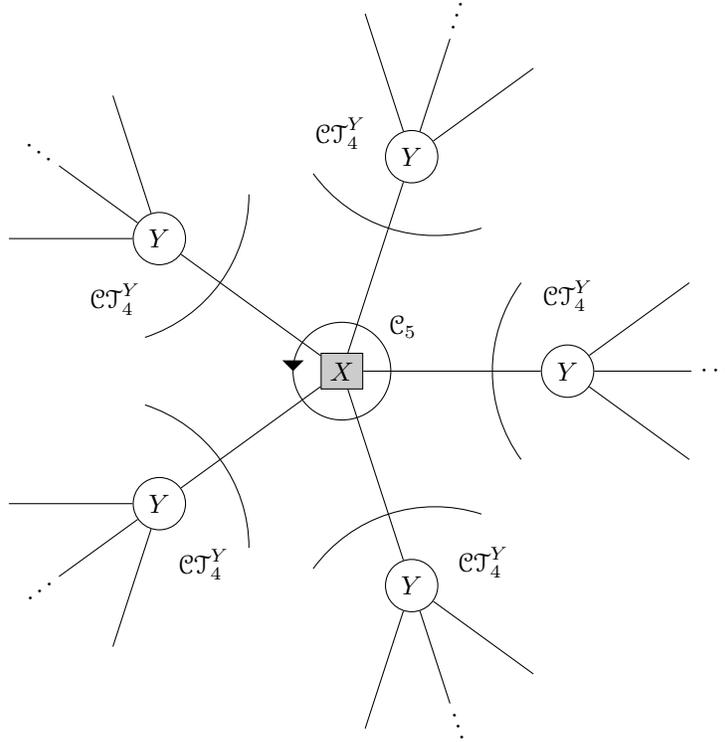
\begin{figure}[htb]
    \centering
    \def\kval{4}
    \begin{tikzpicture}
      \pgfmathparse{int(1+\kval)} %int necessary for clean captions later
      \let\numfronts\pgfmathresult

      \pgfmathparse{180/\numfronts}
      \let\childshift\pgfmathresult

      \node [style=xnode] (root) at (0,0) {$X$};   
      \cycccwl{(root)};

      \draw (180/\numfronts:1) node {$\specname{C}_{\numfronts}$};

      \foreach \i in {0, ..., \kval} {
        \pgfmathparse{360*\i/\numfronts}
        \let\theta\pgfmathresult     
        
        \node [style=ynode] (child\i) at (\theta:3) {$Y$};
        \path [draw] (root) -- (child\i);
        \draw (child\i) ++(\theta+90:1) node {$\cty{\kval}$};

        \draw (child\i) ++(\theta:1cm) ++(180+\theta-\childshift:2cm) arc (180+\theta-\childshift:180+\theta+\childshift:2cm);

        \path [draw] (child\i) -- ++(\theta:2) node [rotate=\theta,fill=white] {$\cdots$};
        \path [draw] (child\i) -- ++(\theta+\childshift:2);
        \path [draw] (child\i) -- ++(\theta-\childshift:2);
      }
      
    \end{tikzpicture}
    \caption[An example $X$-rooted $k$-coding tree]{An example $\ctx{\kval}$-structure, rooted at the $X$-vertex.}
    \label{fig:ctxconst}
  \end{figure}

  Similarly, a $\cty{k}$-structure consists of a single $Y$-label and a (possibly empty) set of structures which are a slight variant of the $\ctx{k}$-structures discussed above.
  Every white neighbor of the black root of a $\ctx{k}$-structure is labeled in the construction above, but the white parent of a $\ctx{k}$-structure in this recursive decomposition is already labeled.
  Thus, the structure around a black vertex which is a child of a white vertex consists of an $X$ label and a linearly-ordered $k$-set of $\cty{k}$-structures.
  Thus, a $\cty{k}$-structure consists of a $Y$-label and a set of pairs of an $X$ label and an $\specname{L}_{k}$-structure of $\cty{k}$-structures.
  This gives \cref{eq:ctyfunc}.
  We note here for conceptual consistency that in fact $\specname{L}_{k} = \deriv{\specname{C}}_{k+1}$ for $\specname{L}$ the species of linear orders and $\specname{C}$ the species of cyclic orders and that $\deriv{\specname{E}} = \specname{E}$ for $\specname{E}$ the species of sets; readers familiar with the $R, S$-enriched bicolored trees of \cite[\S 3.2]{bll:species} will recognize echoes of their decomposition in these facts.

  Finally, a $\ctxy{k}$-structure is simply an $X \cdot \specname{L}_{k} \pbrac[\big]{\cty{k}}$-structure as described above (corresponding to the black vertex) together with a $\cty{k}$-structure (corresponding to the white vertex).
  For reasons that will become clear later, we note that we can incorporate the root white vertex into the linear order by making it last, thus representing a $\ctxy{k}$-structure instead as an $X \cdot \specname{L}_{k+1} \pbrac[\big]{\cty{k}}$-structure.
  This gives \cref{eq:ctxyfunc}.
  See \cref{fig:ctxyconst} for an example of this construction.
  \qedhere

  \begin{figure}[htb]
    \centering
    \def\kval{4}
    \begin{tikzpicture}
      \pgfmathparse{int(1+\kval)} %int necessary for clean captions later
      \let\numfronts\pgfmathresult

      \pgfmathparse{180/\numfronts}
      \let\childshift\pgfmathresult

      \node [style=xnode] (root) at (0,0) {$X$};
      \draw[->,postaction={decorate},decoration={markings, mark = at position 1 with {\arrow{triangle 90}}}] (root) ++(2*\childshift:1cm) arc (2*\childshift:360:1cm); %What a mess! At least it works now...

      \draw (root) -- ++(0:3) [ultra thick] node [fill=white] {};

      \draw (180/\numfronts:1) node {$\specname{L}_{\numfronts}$};

      \foreach \i in {0, ..., \kval} {
        \pgfmathparse{360*\i/\numfronts}
        \let\theta\pgfmathresult     
        
        \node [style=ynode] (child\i) at (\theta:3) {$Y$};
        \path [draw] (root) -- (child\i);
        \draw (child\i) ++(\theta+90:1) node {$\cty{\kval}$};

        \draw (child\i) ++(\theta:1cm) ++(180+\theta-\childshift:2cm) arc (180+\theta-\childshift:180+\theta+\childshift:2cm);

        \path [draw] (child\i) -- ++(\theta:2) node [rotate=\theta,fill=white] {$\cdots$};
        \path [draw] (child\i) -- ++(\theta+\childshift:2);
        \path [draw] (child\i) -- ++(\theta-\childshift:2);
      }
      
    \end{tikzpicture}
    \caption[An example $XY$-rooted $k$-coding tree]{An example $\ctxy{\kval}$-structure, rooted at the $X$-vertex and the thick edge adjoining it.}
    \label{fig:ctxyconst}
  \end{figure}
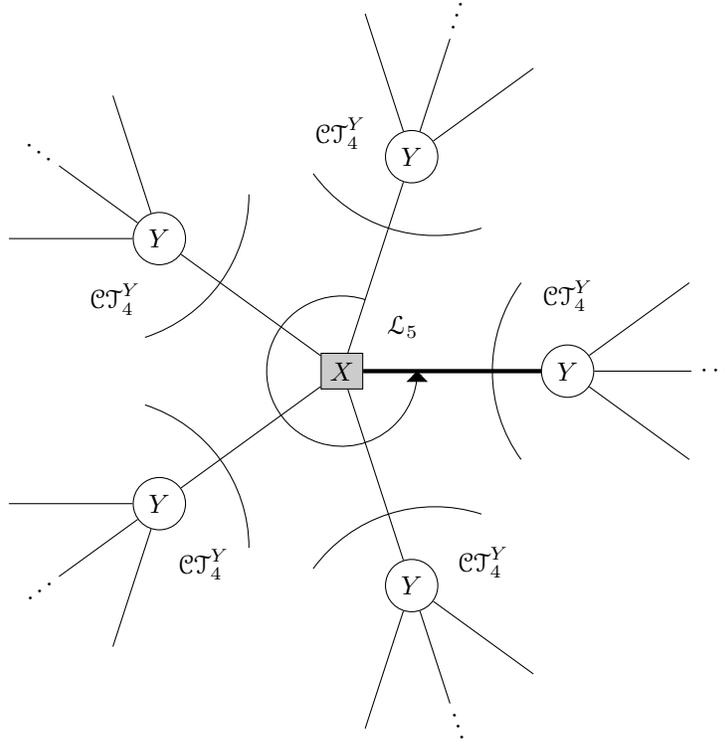
\end{proof}

However, a recursive characterization of the various species of $k$-coding trees is insufficient to characterize the species of $k$-trees itself, since $k$-coding trees represent $k$-trees with coherent orientations.

\section{Generic $k$-trees}\label{s:genkt}
In \cite{gessel:spec2trees}, the orientation-reversing action of $\symgp{2}$ on $\cyc_{\sbrac{3}}$ is exploited to study $2$-trees species-theoretically.
We might hope to develop an analogous group action under which general $k$-trees are naturally identified with orbits of coherently-oriented $k$-trees under an action of $\symgp{k}$.
Unfortunately:
\begin{proposition}
  \label{prop:notransac}
  For $k \geq 3$, no transitive action of any group on the set $\cyc_{\sbrac{k+1}}$ of cyclic orders on $\sbrac{k+1}$ commutes with the action of $\symgp{k+1}$ that permutes labels.
\end{proposition}
\begin{proof}
  We represent the elements of $\cyc_{\sbrac{k+1}}$ as cyclic permutations on the alphabet $\sbrac{k+1}$; then the action of $\symgp{k+1}$ that permutes labels is exactly the conjugation action on these permutations.
  Consider an action of a group $G$ on $\cyc_{\sbrac{k+1}}$ that commutes with this conjugation action.
  Then, for any $g \in G$ and any $c \in \cyc_{\sbrac{k+1}}$, we have that
  \begin{equation}
    \label{eq:transaction}
    g \cdot c = g \cdot c c c^{-1} = c \pbrac{g \cdot c} c^{-1}
  \end{equation}
  and so $c$ and $g \cdot c$ commute.
  Thus, $c$ commutes with every element of its orbit under the action of $G$.
  But, for $k \geq 3$, not all elements of $\cyc_{\sbrac{k+1}}$ commute, so the action is not transitive.
\end{proof}

We thus cannot hope to attack the coherent orientations of $k$-trees by acting directly on the cyclic orderings of fronts.
Instead, we will use the additional structure on \emph{rooted} coherently-oriented $k$-trees; with rooting, the cyclic orders around black vertices are converted into linear orders, for which there is a natural action of $\symgp{k+1}$.

\subsection{Group actions on $k$-coding trees}\label{ss:actct}
We have noted previously that every labeled $k$-tree admits exactly $k!$ coherent orientations.
Thus, there are $k!$ distinct $k$-coding trees associated to each labeled $k$-tree, which differ only in the $\specname{C}_{k+1}$-structures on their black vertices.
Consider a rooted $k$-coding tree $T$ and a black vertex $v$ which is not the root vertex.
Then one white neighbor of $v$ is the `parent' of $v$ (in the sense that it lies on the path from $v$ to the root).
We thus can convert the cyclic order on the $k+1$ white neighbors of $v$ to a linear order by choosing the parent white neighbor to be last.
There is a natural, transitive, label-independent action of $\symgp{k+1}$ on the set of such linear orders which induces an action on the cyclic orders from which the linear orders are derived.
However, only elements of $\symgp{k+1}$ which fix $k+1$ will respect the structure around the black vertex we have chosen, since its parent white vertex must remain last.

In addition, if we simply apply the action of some $\sigma \in \symgp{k+1}$ to the order on white neighbors of $v$, we change the coherently-oriented $k$-tree $\beta^{-1} \pbrac{T}$ to which $T$ is associated in such a way that it no longer corresponds to the same unoriented $k$-tree.
Let $t$ denote the unoriented $k$-tree associated to $\beta^{-1} \pbrac{T}$; then there exists a coherent orientation of $t$ which agrees with orientation around $v$ induced by $\sigma$.
The $k$-coding tree $T'$ corresponding to this new coherent orientation has the same underlying bicolored tree as $T$ but possibly different orders around its black vertices.
If we think of the $k$-coding tree $T'$ as the image of $T$ under a global action of $\sigma$, orbits under all of $\symgp{}$ will be precisely the classes of $k$-coding trees corresponding to all coherent orientations of specified $k$-trees, allowing us to study unoriented $k$-trees as quotients.
The orientation of $T'$ will be that obtained by applying $\sigma$ at $v$ and then recursively adjusting the other cyclic orders so that fronts which were mirror are made mirror again.
This will ensure that the combinatorial structure of the underlying $k$-tree $t$ is preserved.

Therefore, when we apply some permutation $\sigma \in \symgp{k+1}$ to the white neighbors of a black vertex $v$, we must also permute the cyclic orders of the descendant black vertices of $v$.
In particular, the permutation $\sigma'$ which must be applied to some immediate black descendant $v'$ of $v$ is precisely the permutation on the linear order of white neighbors of $v'$ induced by passing over the mirror bijection from $v'$ to $v$, applying $\sigma$, and then passing back.
We can express this procedure in formulaic terms:
\begin{theorem}
  \label{thm:rhodef}
  If a permutation $\sigma \in \symgp{k+1}$ is applied to a linearized orientation of a black vertex $v$ in rooted $k$-coding tree, the permutation which must be applied to the linearized orientation a child black vertex $v'$ which was attached to the $i$th white child of $v$ (with respect to the linear ordering induced by the orientation) to preserve the mirror relation is $\rho_{i} \pbrac{\sigma}$, where $\rho_{i}$ is the map given by
  \begin{equation}
    \label{eq:rhodef}
    \rho_{i} \pbrac{\sigma}: a \mapsto \sigma \pbrac{i + a} - \sigma \pbrac{i}
  \end{equation}
  in which all sums and differences are reduced to their representatives modulo $k+1$ in $\cbrac{1, 2, \dots, k+1}$.
\end{theorem}
\begin{proof}
  Let $v'$ denote a black vertex which is attached to $v$ by the white vertex $1$, which we suppose to be in position $i$ in the linear order induced by the original orientation of $v$.
  Let $2$ denote the white neighbor of $v'$ which is $a$th in the linear order induced by the original orientation around $v'$.
  It is mirror to the white neighbor $3$ of $v$ which is $\pbrac{i+a}$th in the linear order induced by the original orientation around $v$.
  After the action of $\sigma$ is applied, vertex $3$ is $\sigma \pbrac{i+a}$th in the new linear order around $v$.
  We require that $2$ is still mirror to $3$, so we must move it to position $\sigma \pbrac{i + a} - \sigma \pbrac{i}$ when we create a new linear order around $v'$.
  This completes the proof.
\end{proof}
This procedure is depicted in \cref{fig:rhoapp}.

\begin{figure}[htbp]
  \centering
  \begin{tikzpicture}
    \SetGraphUnit{3.5}

    \Vertex[style=xnode,Math=true]{v}
    \cycccwl{(v)};

    \EA[style=ynode](v){3}
    \path (v) edge [bend right=45, thick] node [below](a3d){$i+a$} (3);
    \path (v) edge [bend left=45, thick] node [above](a3u){$\sigma \pbrac{i+a}$} (3);
    \path (a3d) edge [->, dashed, thick] node [auto]{$\sigma$} (a3u);

    \WE[style=ynode](v){1}
    \path (v) edge [bend left=45, thick] node [below](a1d){$i$} (1);
    \path (v) edge [bend right=45, thick] node [above](a1u){$\sigma \pbrac{i}$} (1);
    \path (a1d) edge [->, dashed, thick] node [auto]{$\sigma$} (a1u);
    
    \WE[style=xnode, Math=true](1){v'}
    \cyccwr{(v')};
    \path (v') edge [thick] node [auto](b1){$0$} (1);

    \WE[style=ynode](v'){2}
    \path (v') edge [bend left=45, thick] node [below](b2d){$a$} (2);
    \path (v') edge [bend right=45, thick] node [above](b2u){$\sigma \pbrac{i+a} - \sigma \pbrac{i}$} (2);
    \path (b2d) edge [->, dashed, thick] node [auto]{$\rho_{i} \pbrac{\sigma}$} (b2u);

    \path (b2d) edge [->, dotted, thick, bend right=15] node [auto]{$\mu$} (a3d);
    \path (b2u) edge [->, dotted, thick, bend left=15] node [auto]{$\mu$} (a3u);
  \end{tikzpicture}
  \caption[The permutation-modifying map $\rho$]{Application of a permutation $\sigma$ to the orientation of a non-root black vertex $v$.
    The vertices $2$ and $3$ are mirror in the original orientation (lower set of edges), as shown by the arrows $\mu$, so we must preserve this mirror relation when we apply $\sigma$.
    The permutation $\sigma$ moves $3$ from the $\pbrac{i+a}$th place to the $\sigma \pbrac{i+a}$th, so $\rho_{i} \pbrac{\sigma}$ must carry $2$ from the $a$th place to the $\pbrac{\sigma \pbrac{i+a} - \sigma \pbrac{i}}$th.}
  \label{fig:rhoapp}
\end{figure}
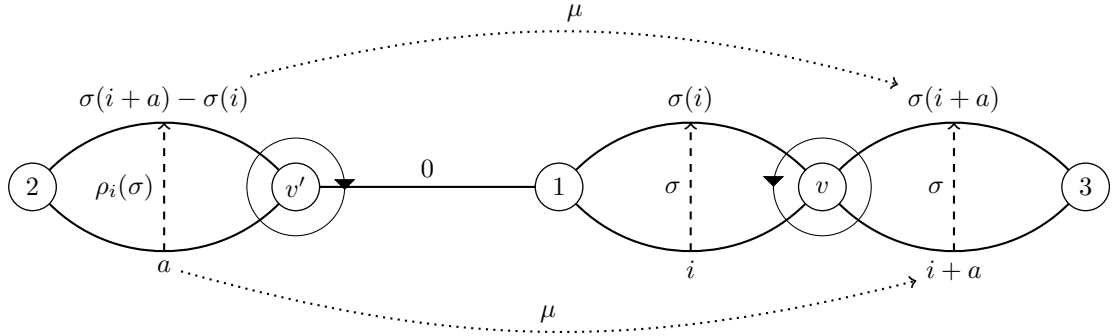

As an aside, we note that, although the construction $\rho$ depends on $k$, the value of $k$ will be fixed in any given context, so we suppress it in the notation.

Any $\sigma$ which is to be applied to a non-root black vertex $v$ must of course fix $k+1$.
We can think of $\symgp{k}$ as the subgroup of $\symgp{k+1}$ of permutations fixing $k+1$, and so in what follows we will refer to this as an $\symgp{k}$-action where appropriate.

In light of \cref{thm:funcdecompct}, we now wish to adapt these ideas into explicit $\symgp{k}$- and $\symgp{k+1}$-actions on $\ctx{k}$, $\cty{k}$, and $\ctxy{k}$ whose orbits correspond to the various coherent orientations of single underlying rooted $k$-trees.
In the case of a $Y$-rooted $k$-coding tree $T$, if we declare that $\sigma \in \symgp{k}$ acts on $T$ by acting directly on each of the black vertices immediately adjacent to the root and then applying $\rho$-derived permutations recursively to their descendants, orbits behave as expected.
The same $\symgp{k}$-action serves equally well for edge-rooted $k$-coding trees, where (for purposes of applying the action of some $\sigma$) we can simply ignore the black vertex in the root.

However, if we begin with an $X$-rooted $k$-coding tree, the cyclic ordering of the white neighbors of the root black vertex has no canonical choice of linearization.
If we make an arbitrary choice of one of the $k+1$ available linearizations, and thus convert to an edge-rooted $k$-coding tree, the full $\symgp{k+1}$-action defined previously can be applied directly to the root vertex.
The orbit under this action of some edge-rooted $k$-coding tree $T$ with a choice of linearization at the root then includes all possible linearizations of the root orders of all possible $X$-rooted $k$-coding trees corresponding to the different coherent orientations of a single $k$-coding tree.

It follows that:
\begin{lemma}
  \label{lem:ctktransact}
  The actions of $\symgp{k}$ on $\cty{k}$ and $\ctxy{k}$ and the action of $\symgp{k+1}$ on $\ctxy{k}$ are transitive in the sense that each orbit corresponds to the set of all coherent orientations of a single underlying rooted unoriented $k$-tree.
  Moreover, these actions all commute with the $\symgp{n}$-actions which permute labels.
\end{lemma}

\subsection{$k$-trees as quotients}\label{ss:ktquot}
We have now equipped the various species of rooted $k$-coding trees with actions which commute with permutations of labels, making them `species-compatible'.
Moreover, the non-oriented rooted $k$-trees underlying the $k$-coding trees are naturally identified with orbits under these actions, suggesting that rooted $k$-trees are `quotients' of rooted $k$-coding trees.
The notion of a $\Gamma$-species as developed in \cite[\S 3]{hend:specfield} formalizes this notion of compatibility and provides an enumerative toolset for dealing with quotients of this sort.
In this language, we will treat $\cty{k}$ and $\ctxy{k}$ as $\symgp{k}$-species and $\ctxy{k}$ as an $\symgp{k+1}$-species\footnote{The $\symgp{k}$- and $\symgp{k+1}$-actions on $\ctxy{k}$ are compatible, but we will make explicit reference to $\ctxy{k}$ as an $\symgp{k}$- or $\symgp{k+1}$-species whenever it is important and not completely clear from context which we mean.} (indicating that they carry equivariant actions of the specified group).
(Hereafter, when it is necessary to distinguish, a species which is \emph{not} equipped with any $\Gamma$-species structure will be dubbed an `actionless' species.)

As a consequence of \cref{lem:ctktransact}, then, we can then relate the rooted $\Gamma$-species forms of $\ct{k}$ to the various (actionless) species forms of generic rooted $k$-trees in \cref{thm:dissymk}:
\begin{theorem}
  \label{thm:arootquot}
  For the various rooted forms of the (actionless) species $\kt{k}$ as in \cref{thm:dissymk} and the various rooted $\Gamma$-species forms of $\ct{k}$ as in \cref{thm:funcdecompct} (interpreted as $\symgp{k}$- and $\symgp{k+1}$-species), we have
  \begin{subequations}
    \label{eq:arootquot}
    \begin{align}
      \kty{k} &= \faktor{\cty{k}}{\symgp{k}} \label{eq:ayquot} \\
      \ktxy{k} &= \faktor{\ctxy{k}}{\symgp{k}} \label{eq:axyquot} \\
      \ktx{k} &= \faktor{\ctxy{k}}{\symgp{k+1}} \label{eq:axquot}
    \end{align}
  \end{subequations}
  as isomorphisms of (actionless) species, where $\ctxy{k}$ is an $\symgp{k}$-species in \cref{eq:axyquot} and an $\symgp{k+1}$-species in \cref{eq:axquot}.
\end{theorem}

As a result, we have explicit characterizations of all the rooted components of the original dissymmetry theorem, \cref{thm:dissymk}.
Thus, through the enumerative toolset of species theory, we can enumerate $k$-trees through a careful enumeration of each of $\cty{k}$ and $\ctxy{k}$.

\section{Automorphisms and cycle indices}\label{s:ktcycind}
Species theory associates to each species $F$ an enumerative power series $\ci{F}$ dubbed the `cycle index' which keeps track of the number of structures with a given automorphism type.
$\Gamma$-species theory provides a natural extension of the cycle index for a $\Gamma$-species $F$, denoted $\gci{\Gamma}{F}$, which keeps track of the number of $\gamma$-invariant structures with a given automorphism type for each $\gamma \in \Gamma$, defined in \cite[\S 3]{hend:specfield}.
We reprint its definition here for convenience:
\begin{equation}
  \label{eq:gcidef}
  \gcielt{\Gamma}{F}{\gamma} \defeq \sum_{n \geq 0} \frac{1}{n!} \sum_{\nu \in \symgp{n}} \fix \pbrac{\gamma \cdot F \sbrac{\nu}} p_{\nu},
\end{equation}
where $p_{\nu} = p_{1}^{\nu_{1}} p_{2}^{\nu_{2}} \dots$ (for $\nu_{i}$ the number of $i$-cycles of $\nu$) is the monomial recording the cycle structure of $\nu$.
(It is important to note that the action of $\symgp{n}$ on labels and the actions of $\symgp{k}$ and $\symgp{k+1}$ on orientations of $k$-coding trees are distinct.
We will refer to label permutations again in the proof of \cref{thm:ctyfuncci}.)

Just as with classical cycle indices (or enumerative generating functions of non-species-theoretic combinatorics), the algebra of $\Gamma$-cycle indices is closely associated to the combinatorial algebra of their species.
We will now apply the results of the preceding sections to compute the cycle indices of the various $\Gamma$-species we have developed.

\subsection{$k$-coding trees: $\cty{k}$ and $\ctxy{k}$}\label{ss:ctcycind}
\Cref{cor:dissymkreform} of the dissymmetry theorem for $k$-trees has a direct analogue in terms of cycle indices:
\begin{theorem}
  \label{thm:dissymkci}
  For the various forms of the species $\kt{k}$ as in \cref{s:dissymk}, we have
  \begin{equation*}
    \label{eq:dissymkci}
    \ci{\kt{k}} = \ci{\ktx{k}} + \ci{\kty{k}} - \ci{\ktxy{k}}.
  \end{equation*}
\end{theorem}

Thus, we need to calculate the cycle indices of the three rooted forms of $\kt{k}$.
A straightforward application of Burnside's lemma allows us to pass from the cycle index of a $\Gamma$-species $F$ to the ordinary cycle index of the quotient species $\nicefrac{\Gamma}{F}$:
\begin{lemma}\label{lem:qsci}
  For a $\Gamma$-species $F$, the ordinary cycle index of the quotient species $\nicefrac{F}{\Gamma}$ is given by 
  \begin{equation*}
    \label{eq:quotcycind}
    \ci{F / \Gamma} = \qgci{\Gamma}{F} \defeq \frac{1}{\abs{\Gamma}} \sum_{\gamma \in \Gamma} \gcielt{\Gamma}{F}{\gamma} = \frac{1}{\abs{\Gamma}} \sum_{\substack{n \geq 0 \\ \nu \in \symgp{n} \\ \gamma \in \Gamma}} \frac{1}{n!} \pbrac{\gamma \cdot F \sbrac{\nu}} p_{\nu}.
  \end{equation*}
  where we define $\qgci{\Gamma}{F} = \frac{1}{\abs{\Gamma}} \sum_{\gamma \in \Gamma} \gcielt{\Gamma}{F}{\gamma}$ for future convenience.
\end{lemma}

From \cref{thm:arootquot} and by \cref{lem:qsci} we obtain:
\begin{theorem}
  \label{thm:aquotci}
  For the various forms of the species $\kt{k}$ as in \cref{s:dissymk} and the various $\symgp{k}$-species and $\symgp{k+1}$-species forms of $\ct{k}$ as in \cref{ss:actct}, we have
  \begin{subequations}
    \label{eq:aquotci}
    \begin{align}
      \ci{\kty{k}} &= \qgci{\symgp{k}}{\cty{k}} = \frac{1}{k!} \cdot \smashoperator{\sum_{\sigma \in \symgp{k}}} \gcielt{\symgp{k}}{\cty{k}}{\sigma} \label{eq:ayquotci} \\
      \ci{\ktxy{k}} &= \qgci{\symgp{k}}{\ctxy{k}} = \frac{1}{k!} \cdot \smashoperator{\sum_{\sigma \in \symgp{k}}} \gcielt{\symgp{k}}{\ctxy{k}}{\sigma} \label{eq:axyquotci} \\
      \ci{\ktx{k}} &= \qgci{\symgp{k+1}}{\ctxy{k}} = \frac{1}{\pbrac{k+1}!} \cdot \smashoperator{\sum_{\sigma \in \symgp{k+1}}} \gcielt{\symgp{k+1}}{\ctxy{k}}{\sigma} \label{eq:axquotci}
    \end{align}
  \end{subequations}
\end{theorem}

We thus need only calculate the various $\Gamma$-cycle indices for the $\symgp{k}$-species and $\symgp{k+1}$-species forms of $\cty{k}$ and $\ctxy{k}$ to complete our enumeration of general $k$-trees.

In \cref{thm:funcdecompct}, the functional equations for the (actionless) species $\cty{k}$ and $\ctxy{k}$ both include terms of the form $\specname{L}_{k} \circ \cty{k}$.
The plethysm of (actionless) species does have a generalization to $\Gamma$-species, as given in \cite[\S 3]{hend:specfield}, but it does not correctly describe the manner in which $\symgp{k}$ acts on linear orders of $\cty{k}$-structures in these recursive decompositions.
Specifically, for two $\Gamma$-species $F$ and $G$, an element $\gamma \in \Gamma$ acts on an $\pbrac{F \circ G}$-structure (colloquially, `an $F$-structure of $G$-structures') by acting on the $F$-structure and on each of the $G$-structures independently.
In our action of $\symgp{k}$, however, the actions of $\sigma$ on the descendant $\cty{k}$-structures are \emph{not} independent---they depend on the position of the structure in the linear ordering around the parent black vertex.
In particular, if $\sigma$ acts on some non-root black vertex, then $\rho_{i} \pbrac{\sigma}$ acts on the white vertex in the $i$th place, where in general $\rho_{i} \pbrac{\sigma} \neq \sigma$.

Thus, we consider automorphisms of these $\symgp{k}$-structures directly.
First, we consider the component species $X \cdot \specname{L}_{k} \pbrac[\big]{\cty{k}}$.
\begin{lemma}
  \label{lem:ctyinvar}
  Let $B$ be a structure of the species $F_{k} = X \cdot \specname{L}_{k} \pbrac[\big]{\cty{k}}$.
  Let $W_{i}$ be the $\cty{k}$-structure in the $i$th position in the linear order.
  Then some $\sigma \in \symgp{k}$ acts as an automorphism of $B$ if and only if, for each $i \in \sbrac{k+1}$, we have $\pbrac*{\rho_{i} \sigma} W_{i} \cong W_{\sigma \pbrac{i}}$.
\end{lemma}

\begin{proof}
  Recall that the action of $\sigma \in \symgp{k}$ is in fact the action of the lift of $\sigma$ as an element of $\symgp{k+1}$.
  The $X$-label on the black root of $B$ is not affected by the action of $\sigma$, so no conditions on $\sigma$ are necessary to accommodate it.
  However, the $\specname{L}_{k}$-structure on the white children of the root is permuted by $\sigma$, and we apply to each of the $W_{i}$'s the action of $\pbrac*{\rho_{i} \sigma}$.
  (See \cref{fig:fkconst}.)
  Thus, $\sigma$ is an automorphism of $B$ if and only if the combination of applying $\sigma$ to the linear order and $\rho_{i} \sigma$ to each $W_{i}$ is an automorphism.
  Since $\sigma$ `carries' each $W_{i}$ onto $W_{\sigma \pbrac{i}}$, we must have that $\pbrac*{\rho_{i} \sigma} W_{i} \cong W_{\sigma \pbrac{i}}$, as claimed.
  That this suffices is clear. \qedhere

  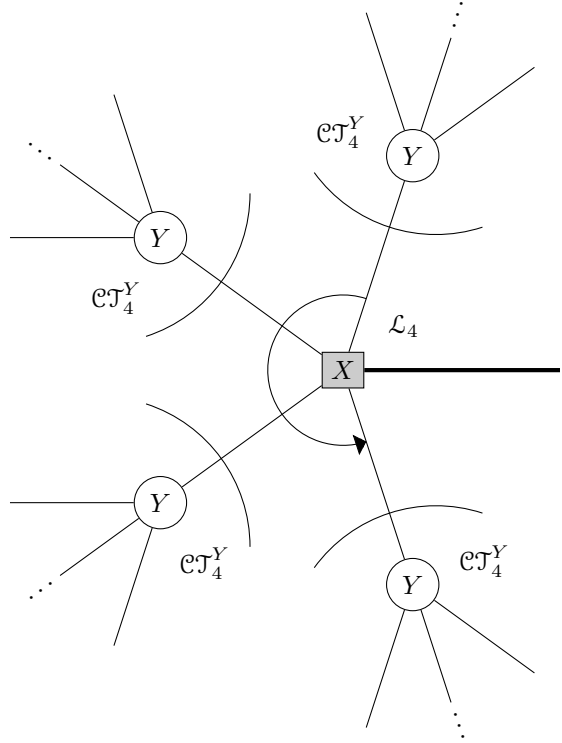
\begin{figure}[htb]
    \centering
    \def\kval{4}
    \begin{tikzpicture}
      \pgfmathparse{int(1+\kval)} %int necessary for clean captions later
      \let\numfronts\pgfmathresult

      \pgfmathparse{180/\numfronts}
      \let\childshift\pgfmathresult

      \node [style=xnode] (root) at (0,0) {$X$};
      \draw[->,postaction={decorate},decoration={markings, mark = at position 1 with {\arrow{triangle 90}}}] (root) ++(2*\childshift:1cm) arc (2*\childshift:360*(1-1/\numfronts):1cm); %What a mess! At least it works now...

      \draw (root) -- ++(0:3) [ultra thick] node [fill=white] {};

      \draw (180/\numfronts:1) node {$\specname{L}_{\kval}$};

      \foreach \i in {1, ..., \kval} {
        \pgfmathparse{360*\i/\numfronts}
        \let\theta\pgfmathresult     
        
        \node [style=ynode] (child\i) at (\theta:3) {$Y$};
        \path [draw] (root) -- (child\i);
        \draw (child\i) ++(\theta+90:1) node {$\cty{\kval}$};

        \draw (child\i) ++(\theta:1cm) ++(180+\theta-\childshift:2cm) arc (180+\theta-\childshift:180+\theta+\childshift:2cm);

        \path [draw] (child\i) -- ++(\theta:2) node [rotate=\theta,fill=white] {$\cdots$};
        \path [draw] (child\i) -- ++(\theta+\childshift:2);
        \path [draw] (child\i) -- ++(\theta-\childshift:2);
      }
      
    \end{tikzpicture}
    \caption[An example $F_{k}$-structure]{An example $F_{\kval}$-structure.}
    \label{fig:fkconst}
  \end{figure}
\end{proof}

We hereafter treat $F_{k}$ as a $\symgp{k}$-species with respect to this action, but note that $F_{k} = X \cdot \specname{L}_{k} \pbrac[\big]{\cty{k}}$ is \emph{not} an isomorphism of $\symgp{k}$-species.

We now have the tools in hand to find recursive functional equations satisfied by $\gci{\symgp{k}}{\cty{k}}$ and $\gci{\symgp{k}}{\ctxy{k}}$.

\begin{theorem}[Cycle index of $\cty{k}$]
  \label{thm:ctyfuncci}
  The $\symgp{k}$-cycle index for the species $\cty{k}$ is characterized by the recursive functional equations
  \begin{subequations}
    \label{eq:ctyfuncci}
    \begin{gather}
      \label{eq:ctyfunccigci}
      \gci{\symgp{k}}{\cty{k}} = \gci{\symgp{k}}{Y} \cdot \pbrac*{\gci{\symgp{k}}{\specname{E}} \circ \gci{\symgp{k}}{F_{k}}} \\
      \label{eq:ctyfunccif}
      \gcielt{\symgp{k}}{F_{k}}{\sigma} = p_{1} \sbrac{x} \cdot \smashoperator{\prod_{c \in C \pbrac{\sigma}}} \gci{\symgp{k}}{\cty{k}} \pbrac[\Big]{\prod_{i \in c} \rho_{i} \pbrac{\sigma}} \pbrac*{p_{\abs{c}} \sbrac{x}, p_{2 \abs{c}} \sbrac{x}, \dots; p_{\abs{c}} \sbrac{y}, p_{2 \abs{c}} \sbrac{y}, \dots}.
    \end{gather}
  \end{subequations}
  In \cref{eq:ctyfunccigci}, the plethysm $\circ$ is that of $\symgp{k}$-species, $\specname{E}$ is the $\symgp{k}$-species of sets with the trivial $\symgp{k}$-action, and $Y$ is the $\symgp{k}$-species of $Y$-labeled singletons with the trivial $\symgp{k}$-action.
  In \cref{eq:ctyfunccif}, $C \pbrac{\sigma}$ denotes the set of cycles of $\sigma$ (as a $k$-permutation), and the product is taken in order with respect to any choice of linearization of the cyclic order of the elements of $c$.
\end{theorem}
\begin{proof}
  Let $T$ be a structure of the $\symgp{k}$-species $\cty{k}$.
  It is clear that $T$ may be decomposed into a singleton of type $Y$, corresponding to the label on the root white vertex, and an arbitrary set ($\specname{E}$-structure) of descendant $X$-rooted trees, all of the same ``structure type'' (that is, species), which we denote $F_{k}$.
  Furthermore, it is apparent that some $\sigma \in \symgp{k}$ acts on $T$ by acting independently on $Y$ (by fixing the root), on $\specname{E}$ (by preserving the set of neighbors of that root), and on each $F_{k}$-structure (in some more complicated manner, to be discussed in what follows.
  \Cref{eq:ctyfunccigci} follows immediately from these observations.
  It remains only to analyze the $\symgp{k}$-species $F_{k}$ of these $X$-rooted descendant structures.
  
  By \cref{lem:ctyinvar}, if $\sigma \cdot \cty{k} \sbrac{\pi, \tau}$ fixes some $F_{k}$ structure, the $\rho_{i} \pbrac{\sigma}$-image of the $\cty{k}$-structure at position $i$ must be its $\pbrac*{\cty{k} \sbrac{\pi, \tau}}^{-1}$-image also.
  Furthermore, if $i$ is in a cycle $c$ of length $\abs{c}$, then all the other $\cty{k}$-structures along $c$ are determined by the choice of the structure at position $i$, and that $\cty{k}$-structure must be sent to \emph{itself} by $\prod_{j \in c} \rho_{j} \pbrac{\sigma}$.
  The action of this permutation on that structure must then be identical to that of $\pbrac*{\cty{k} \sbrac{\pi, \tau}}^{-\abs{c}}$, which we observe must then restrict to an automorphism of that $\cty{k}$-structure, each $l$-cycle of which corresponds to an $\pbrac{l \cdot \abs{c}}$-cycle of $\cty{k} \sbrac{\pi, \tau}$.

  Now we consider the $\symgp{k}$-cycle index $\gci{\symgp{k}}{F_{k}}$ of this $\symgp{k}$-species $F_{k}$.
  Fix partitions $\pi$ (representing the cycle type of a permutation acting on black ($x$-) vertex labels) and $\tau$ (similarly for white ($y$-) labels) and some $\sigma \in \symgp{k}$.
  (In what follows, we let $\hat{\lambda}$ denote some arbitrarily-chosen permutation of cycle type $\lambda$ when needed.)
  Then the coefficient on the $p_{\pi} \sbrac{x} p_{\tau} \sbrac{y}$-term of $\gcielt{\symgp{k}}{F_{k}}{\sigma}$ is $\frac{1}{\abs{\pi}! \abs{\tau}!}$ times the number of $F_{k}$-structures for which $\sigma \cdot \cty{k} \sbrac{\hat{\pi}, \hat{\tau}}$ is an automorphism.
  Let $f(\pi, \tau)$ denote this number.

  Suppose that $f(\pi, \tau)$ is nonzero.
  By the above, there must exist decompositions\footnote{By `decomposition' of a partition we mean a partition of it as a multiset---i.e.\ a choice of sub-multisets whose pairwise intersections are trivial and whose union is the original partition.} $\pi - \cbrac{1} = \bigsqcup_{i} \pi_{i}$ and $\tau = \bigsqcup_{i} \tau_{i}$, each having $\abs{C \pbrac{\sigma}}$ components, such that all the parts of $\pi_{i}$ and $\tau_{i}$ are multiples of the length of the $i$th cycle of $\sigma$.
  Let $\frac{\pi_{i}}{\abs{c_{i}}}$ and $\frac{\tau_{i}}{\abs{c_{i}}}$ denote the partitions resulting from dividing each part of $\pi_{i}$ and $\tau_{i}$ respectively by $\abs{c_{i}}$.
  Then, for each such decomposition, an $F_{k}$-structure may be assembled by choosing, for each cycle $c_{i}$, a $\cty{k}$-structure for which $\pbrac*{\prod_{j \in c_{i}} \rho_{j} \pbrac{\sigma}} \cdot \cty{k} \sbrac{\frac{\hat{\pi}_{i}}{\abs{c_{i}}}, \frac{\hat{\tau}_{i}}{\abs{c_{i}}}}$ is an automorphism, then distributing its $\rho$-images around the $\abs{c_{i}}$ positions of the cycle.
  The number of such choices is exactly $\abs{\frac{\pi_{i}}{\abs{c_{i}}}}! \abs{\frac{\tau_{i}}{\abs{c_{i}}}}!$ times the coefficient of the $p_{\frac{\hat{\pi}_{i}}{\abs{c_{i}}}} \sbrac{x} p_{\frac{\hat{\tau}_{i}}{\abs{c_{i}}}} \sbrac{y}$-term of $\gcielt{\symgp{k}}{\cty{k}}{\prod_{j \in c_{i}} \rho_{j} \pbrac{\sigma}}$.
  Accordingly, the total number of $F_{k}$-structures for which $\sigma \cdot \cty{k} \sbrac{\hat{\pi}, \hat{\tau}}$ is an automorphism is exactly $\abs{\pi}! \abs{\tau}!$ times the coefficient of $p_{\pi - \sbrac{1}} \sbrac{x} p_{\tau} \sbrac{y}$ in $\prod_{c \in C \sbrac{\sigma}} \gcielt{\symgp{k}}{\cty{k}}{\prod_{i \in c} \rho_{i} \pbrac{\sigma}} \pbrac*{p_{\abs{c}} \sbrac{x}, p_{2 \abs{c}} \sbrac{x}, \dots; p_{\abs{c} \sbrac{y}, p_{2 \abs{c}} \sbrac{y}, \dots}}$.
    \Cref{eq:ctyfunccif} follows immediately.
\end{proof}

\begin{theorem}[Cycle index of $\ctxy{k}$]
  \label{thm:ctxyfuncci}
  The $\symgp{k+1}$-cycle index for the species $\ctxy{k}$ is given by
  \begin{equation}
    \label{eq:ctxyfuncci}
    \gcielt{\symgp{k+1}}{\ctxy{k}}{\sigma} = p_{1} \sbrac{x} \cdot \smashoperator{\prod_{c \in C \pbrac{\sigma}}} \gci{\symgp{k}}{\cty{k}} \pbrac[\Big]{\prod_{i \in c} \rho_{i} \sbrac{\sigma}} \pbrac*{p_{\abs{c}} \sbrac{x}, p_{2 \abs{c}} \sbrac{x}, \dots, p_{\abs{c}} \sbrac{y}, p_{2 \abs{c}} \sbrac{y}, \dots}.
  \end{equation}
  under the same conditions as \cref{thm:ctyfuncci}.
\end{theorem}
\begin{proof}
  The proof is essentially identical to the part of that of \cref{thm:ctyfuncci} which concerns \cref{eq:ctyfunccif}.
\end{proof}

Terms of the form $\prod_{i \in c} \rho_{i} \pbrac{\sigma}$ appear in \cref{eq:ctyfuncci,eq:ctxyfuncci}.
For the simplification of calculations, we note here two useful results about these products.

First, we observe that certain $\rho$-maps preserve cycle structure:
\begin{lemma}
  \label{lem:rhofp}
  Let $\sigma \in \symgp{k}$ be a permutation of which $i \in \sbrac{k}$ is a fixed point.
  Then $\rho_{i} \pbrac{\sigma}$ has the same cycle type as $\sigma$.
\end{lemma}
\begin{proof}
  Let $\theta_{i} \in \symgp{k}$ denote the permutation given by $\theta_{i} \pbrac{a} = a + i$ reduced modulo $k+1$ as in the proof of \cref{thm:rhodef}.
  Then, since $i$ is a fixed point of $\sigma$, we have that, for each $a \in \cbrac{1, 2, \dots, k+1}$,
  \begin{equation*}
    \rho_{i} \pbrac{\sigma} \pbrac{a} = \sigma \pbrac{i+a} - i = \pbrac{\theta_{i} \sigma \theta_{i}^{-1}} \pbrac{a},
  \end{equation*}
  so $\rho_{i} \pbrac{\sigma}$ is a conjugate of $\sigma$ and thus has the same cycle structure as $\sigma$, as desired.
\end{proof}

But then we note that the products in the above theorems are in fact permutations obtained by applying such $\rho$-maps:
\begin{lemma}
  \label{lem:rhoprod}
  Let $\sigma \in \symgp{k}$ be a permutation with a cycle $c$.
  Then the cycle type of $\pbrac*{\prod_{i \in c} \rho_{i} \pbrac{\sigma}}$ is the same as that of $\sigma^{\abs{c}}$.
  In other words, $\lambda \pbrac*{\prod_{i \in c} \rho_{i} \pbrac{\sigma}} = \lambda \pbrac[\big]{\sigma^{\abs{c}}}$.
\end{lemma}
\begin{proof}
  Let $c = \pbrac{c_{1}, c_{2}, \dots, c_{\abs{c}}}$.
  First, we calculate:
  \begin{align*}
    \prod_{i = 1}^{\abs{c}} \rho_{c_{i}} \pbrac{\sigma} =& \rho_{c_{\abs{c}}} \pbrac{\sigma} \circ \dots \circ \rho_{c_{2}} \pbrac{\sigma} \circ \rho_{c_{1}} \pbrac{\sigma} \\
  =& \rho_{c_{\abs{c}}} \pbrac{\sigma} \circ \dots \circ \rho_{c_{2}} \pbrac{\sigma} \pbrac{a \mapsto \sigma \pbrac{c_{1} + a} - \sigma \pbrac{c_{1}}} \\
  =& \rho_{c_{\abs{c}}} \pbrac{\sigma} \circ \dots \circ \rho_{c_{3}} \pbrac{\sigma} \pbrac{a \mapsto \sigma \pbrac{c_{2} + \sigma \pbrac{c_{1} + a} - \sigma \pbrac{c_{1}}} - \sigma \pbrac{c_{2}}} \\
  =& \rho_{c_{\abs{c}}} \pbrac{\sigma} \circ \dots \circ \rho_{c_{3}} \pbrac{\sigma} \pbrac{a \mapsto \sigma^{2} \pbrac{c_{1} + a} - \sigma^{2} \pbrac{c_{1}}} \\
  &\vdots \\
  =& a \mapsto \sigma^{\abs{c}} \pbrac{c_{1} + a} - \sigma^{\abs{c}} \pbrac{c_{1}} \\
  =& \rho_{c_{1}} \pbrac{\sigma^{\abs{c}}}.
  \end{align*}
  But $c_{1}$ is a fixed point of $\sigma^{\abs{c}}$, so by the result of \cref{lem:rhofp}, this has the same cycle structure as $\sigma^{\abs{c}}$.
\end{proof}
As a result, we can bypass all calculations of $\rho_{i} \pbrac{\sigma}$ in the computation of terms of the cycle indices we have developed.

We also note an important general fact about the cycle indices of $\Gamma$-species:
\begin{lemma}
  \label{lem:ciclassfunc}
  Let $F$ be a $\Gamma$-species.
  Then $\gcielt{\Gamma}{F}{\gamma}$ is a class function of $\gamma$.
\end{lemma}

This will simplify computational enumeration of $k$-trees significantly when $k$ is large, since the number of elements of $\symgp{k}$ is factorial in $k$ while the number of conjugacy classes (indexed by partitions) is exponential in $k$.

\subsection{$k$-trees: $\kt{k}$}\label{ss:ktcycind}
We now have all the pieces in hand to apply \cref{thm:dissymkci} to compute the cycle index of the species $\kt{k}$ of general $k$-trees.
\Cref{eq:dissymkci} characterizes the cycle index of the generic $k$-tree species $\kt{k}$ in terms of the cycle indices of the rooted species $\ktx{k}$, $\kty{k}$, and $\ktxy{k}$; \cref{thm:arootquot} gives the cycle indices of these three rooted species in terms of the $\Gamma$-cycle indices $\gci{\symgp{k}}{\cty{k}}$, $\gci{\symgp{k}}{\ctxy{k}}$, and $\gci{\symgp{k+1}}{\ctxy{k}}$; and, finally, \cref{thm:ctyfuncci,thm:ctxyfuncci} give these $\Gamma$-cycle indices explicitly.
By tracing the formulas in \cref{eq:ctyfuncci,eq:ctxyfuncci} back through this sequence of functional relationships, we can conclude:
\begin{theorem}[Cycle index for the species of $k$-trees]
  \label{thm:ktci}
  For $\mathfrak{a}_{k}$ the species of general $k$-trees, $\gci{\symgp{k}}{\cty{k}}$ as in \cref{eq:ctyfuncci}, and $\gci{\symgp{k+1}}{\ctxy{k}}$ as in \cref{eq:ctxyfuncci} we have:
  \label{thm:ktreecyc}
  \begin{subequations}
    \label{eq:ktci}
    \begin{align}
      \ci{\kt{k}} &= \frac{1}{\pbrac{k+1}!} \cdot \smashoperator{\sum_{\sigma \in \symgp{k+1}}} \gcielt{\symgp{k+1}}{\ctxy{k}}{\sigma} + \frac{1}{k!}\cdot \smashoperator{\sum_{\sigma \in \symgp{k}}} \gcielt{\symgp{k}}{\cty{k}}{\sigma} - \frac{1}{k!} \cdot \smashoperator{\sum_{\sigma \in \symgp{k}}} \gcielt{\symgp{k}}{\ctxy{k}}{\sigma} \label{eq:ktciexplicit} \\
      &= \qgci{\symgp{k+1}}{\ctxy{k}} + \qgci{\symgp{k}}{\cty{k}} - \qgci{\symgp{k}}{\ctxy{k}}. \label{eq:ktciquot}
    \end{align}
  \end{subequations}

\end{theorem}

\Cref{eq:ktci} in fact represents a recursive system of functional equations, since the formulas for the $\Gamma$-cycle indices of $\cty{k}$ and $\ctxy{k}$ are recursive.
Computational methods can yield explicit enumerative results.
However, a bit of care will allow us to reduce the computational complexity of this problem significantly.

\section{Unlabeled enumeration and the generating function $\tilde{\mathfrak{a}}_{k} \pbrac{x}$}\label{ss:ktunlenum}
\Cref{eq:ktci} in \cref{thm:ktci} gives a recursive formula for the cycle index of the (actionless) species $\kt{k}$ of $k$-trees.
The number of unlabeled $k$-trees with $n$ hedra is historically an open problem, but it is straightforward to extract their ordinary generating function from the cycle index $\ci{\kt{k}}$ once it is computed.
Actually computing terms of the cycle index in order to derive the coefficients of the generating function is, however, a computationally expensive process, since the cycle index is by construction a power series in two infinite sets of variables.
The computational process can be simplified significantly by taking advantage of the relatively straightforward combinatorial structure of the structural decomposition used to derive the recursive formulas for the cycle index.

For a $\Gamma$-species $F$, the ordinary generating function $\tilde{F}_{\gamma} \pbrac{x}$ counting unlabeled $\gamma$-invariant $F$-structures is given by
\[\tilde{F} \pbrac[\big]{\gamma} \pbrac{x} = \gcieltvars{\Gamma}{F}{\gamma}{x, x^2, x^3, \dots}\]
and the ordinary generating function for counting unlabeled $\nicefrac{F}{\Gamma}$-structures is given by
\[\tilde{F} \pbrac{x} = \frac{1}{\abs{\Gamma}} \sum_{\gamma \in \Gamma} \tilde{F} \pbrac[\big]{\gamma} \pbrac{x}.\]
These formula admits an obvious multisort extension, but we in fact wish to count $k$-trees with respect to just one sort of label (the $X$-labels on hedra), so we will not deal with multisort issues here.
Each of our two-sort cycle indices can be converted to one-sort by substituting $p_{i} \sbrac{y} = 1$ for all $i$.
For the rest of this section, we will deal directly with these one-sort versions of the cycle indices.

We begin by considering the explicit recursive functional equations in \cref{thm:ctyfuncci,thm:ctxyfuncci}.
In each case, by the above, the ordinary generating function is exactly the result of substituting $p_{i} \sbrac{x} = x^{i}$ and $p_{i} \sbrac{y} = 1$ into the given formula.
Thus, we have:
\begin{theorem}
  \label{thm:ctrhoogf}
  For $\cty{k}$ the $\symgp{k}$-species of $Y$-rooted $k$-coding trees and $\ctxy{k}$ the $\symgp{k+1}$-species of edge-rooted $k$-coding trees, the corresponding single-variable $\Gamma$-ordinary generating functions are given by
  \begin{subequations}
    \label{eq:ctrhoogf}
    \begin{equation}
      \widetilde{\cty{k}} \pbrac{\sigma} \pbrac{x} = \exp \pbrac[\Big]{ \sum_{n \geq 1} \pbrac[\Big]{ \frac{x^{n}}{n} \cdot \smashoperator{\prod_{c \in C \pbrac{\sigma^{n}}}} \widetilde{\cty{k}} \pbrac[\Big]{\prod_{i \in c} \rho_{i} \pbrac{\sigma^{n}}} \pbrac{x^{n \abs{c}}}}} \label{eq:ctyrhoogf}
    \end{equation}
    and
    \begin{equation}
      \widetilde{\ctxy{k}} \pbrac{\sigma} \pbrac{x} = x \cdot \smashoperator{\prod_{c \in C \pbrac{\sigma}}} \widetilde{\cty{k}} \pbrac[\Big]{\prod_{i \in c} \rho_{i} \pbrac{\sigma}} \pbrac[\big]{x^{\abs{c}}}. \label{eq:ctxyrhoogf}
    \end{equation}
  \end{subequations}
  where $\widetilde{\cty{k}}$ is an $\symgp{k}$-generating function and $\widetilde{\ctxy{k}}$ is an $\symgp{k+1}$-generating function.
\end{theorem}

However, as a consequence of \cref{lem:rhoprod,lem:ciclassfunc}, we can simplify these expressions significantly.
Let $\lambda$ be the function mapping each permutation to its cycle type interpreted as an integer partition.
Then:
\begin{corollary}
  \label{cor:ctogf}
  For $\cty{k}$ the $\symgp{k}$-species of $Y$-rooted $k$-coding trees and $\ctxy{k}$ the $\symgp{k+1}$-species of edge-rooted $k$-coding trees, the corresponding single-variable $\Gamma$-ordinary generating functions are given by
  \begin{subequations}
    \label{eq:ctogf}
    \begin{equation}
      \widetilde{\cty{k}} \pbrac{\lambda} \pbrac{x} = \exp \pbrac[\Big]{\sum_{n \geq 1} \pbrac[\Big]{ \frac{x^{n}}{n} \cdot \prod_{i \in \lambda^{n}} \widetilde{\cty{k}} \pbrac[\big]{\lambda^{n i}} \pbrac[\big]{x^{n i}}}} \label{eq:ctyogf}
    \end{equation}
    and
    \begin{equation}
      \widetilde{\ctxy{k}} \pbrac{\lambda} \pbrac{x} = x \cdot \prod_{i \in \lambda} \widetilde{\cty{k}} \pbrac[\big]{\lambda^{i} - \cbrac{1}} \pbrac[\big]{x^{i}} \label{eq:ctxyogf}
    \end{equation}
  \end{subequations}
  where $\prod_{i \in \lambda}$ denotes a product over the parts $i$ of $\lambda$ taken with multiplicity, where $\lambda^{i}$ denotes the $i$th `partition power' of $\lambda$ --- that is, if $\sigma$ is any permutation of cycle type $\lambda$, then $\lambda^{i}$ denotes the cycle type of $\sigma^{i}$ --- and where $f \pbrac{\lambda} \pbrac{x}$ denotes the value of $f \pbrac{\sigma} \pbrac{x}$ for any $\sigma$ of cycle type $\lambda$.
\end{corollary}

To clarify the notation, we work out the case $k = 2$ more explicitly here:
\begin{example}[$k = 2$]
  There are only two partitions of $k = 2$: $\cbrac{1, 1}$ and $\cbrac{2}$, corresponding to the two permutations $\pmt{(1)(2)}$ and $\pmt{(12)}$ respectively.
  Since $\cbrac{1, 1}^{i} = \cbrac{1, 1}$ for any $i$, we have that
  \[ \widetilde{\cty{2}} \pbrac{\cbrac{1, 1}} \pbrac{x} = \exp \pbrac*{ \sum_{n \geq 1} \frac{x^{n}}{n} \widetilde{\cty{2}} \pbrac{\cbrac{1, 1}} \pbrac{x^{n}}^{2}}. \]

  The situation for $\cbrac{2}$ is slightly more complex, since $\cbrac{2}^{i} = \cbrac{1, 1}$ if $i$ is even but $\cbrac{2}$ if $i$ is odd.
  Thus, we have
  \[ \widetilde{\cty{2}} \pbrac{\cbrac{2}} \pbrac{x} = \exp \pbrac*{ \sum_{n \geq 1} \pbrac*{\frac{x^{2n-1}}{2n-1} \widetilde{\cty{2}} \pbrac{\cbrac{2}} \pbrac{x^{2n-1}}^{2}} + \pbrac*{\frac{x^{2n}}{2n} \widetilde{\cty{2}} \pbrac{\cbrac{1, 1}} \pbrac{x^{2n}}^{2}}}. \]

  Conventional computational techniques (such as those demonstrated in \cref{c:code}) then suffice to compute that
  \begin{gather*}
    \widetilde{\cty{2}} \pbrac{\cbrac{1, 1}} \pbrac{x} = 1 + x + 3x^{2} + 10x^{3} + 39x^{4} + 160x^{5} + \dots \\
    \intertext{and}
    \widetilde{\cty{2}} \pbrac{\cbrac{2}} \pbrac{x} = 1 + x + x^{2} + 2x^{3} + 3x^{4} + 6x^{5} + \dots.
  \end{gather*}
\end{example}

\Cref{cor:ctogf} characterizes the ordinary generating functions $\widetilde{\cty{k}}$ and $\widetilde{\ctxy{k}}$.
The cycle index of the species $\kt{k}$, as seen in \cref{eq:ktci}, is given simply in terms of quotients of the cycle indices of the two $\Gamma$-species $\cty{k}$ and $\ctxy{k}$, and this result can pass to the generating-function level.
Thus, we also have:
\begin{theorem}
  \label{thm:akrhoogf}
  For $\kt{k}$ the species of $k$-trees and $\widetilde{\cty{k}}$ and $\widetilde{\ctxy{k}}$ as in \cref{thm:ctrhoogf}, we have
  \begin{equation}
    \label{eq:akrhoogf}
    \tilde{\mathfrak{a}}_{k} \pbrac{x} = \frac{1}{\pbrac{k+1}!} \cdot \smashoperator{\sum_{\sigma \in \symgp{k+1}}} \widetilde{\ctxy{k}} \pbrac{\sigma} \pbrac{x} + \frac{1}{k!} \cdot \smashoperator{\sum_{\sigma \in \symgp{k}}} \widetilde{\cty{k}} \pbrac{\sigma} \pbrac{x} - \frac{1}{k!} \cdot \smashoperator{\sum_{\sigma \in \symgp{k}}} \widetilde{\ctxy{k}} \pbrac{\sigma} \pbrac{x}.
  \end{equation}
\end{theorem}

Then, as a consequence of \cref{lem:ciclassfunc,cor:ctogf}, we can instead write
\begin{corollary}[Generating function for unlabeled $k$-trees]
  \label{cor:ktgf}
  For $\kt{k}$ the species of $k$-trees and $\widetilde{\cty{k}}$ and $\widetilde{\ctxy{k}}$ as in \cref{cor:ctogf}, we have
  \begin{equation}
    \label{eq:ktgf}
    \tilde{\mathfrak{a}}_{k} \pbrac{x} = \sum_{\lambda \vdash k+1} \frac{1}{z_{\lambda}} \widetilde{\ctxy{k}} \pbrac{\lambda} \pbrac{x} + \sum_{\lambda \vdash k} \frac{1}{z_{\lambda}} \widetilde{\cty{k}} \pbrac{\lambda} \pbrac{x} - \sum_{\lambda \vdash k} \frac{1}{z_{\lambda}} \widetilde{\ctxy{k}} \pbrac{\lambda \cup \cbrac{1}} \pbrac{x}.
  \end{equation}
\end{corollary}

This direct characterization of the ordinary generating function of unlabeled $k$-trees, while still recursive, is much simpler computationally than the characterization of the full cycle index in \cref{eq:ktci}.
For computation of the number of unlabeled $k$-trees, it is therefore much preferred.
Classical methods for working with recursively-defined power series suffice to extract the coefficients quickly and efficiently.
The results of some such explicit calculations are presented in \cref{c:enum}.

\section{Special-case behavior for small $k$}
Many of the complexities of the preceding analysis apply only for $k \geq 3$.
In the cases $k = 1$ and $k = 2$, our analysis simplifies dramatically, and effectively reduces to previous work.

\subsection{Ordinary trees ($k = 1$)}
When $k = 1$, an $\kt{k}$-structure is merely an ordinary tree with $X$-labels on its edges and $Y$-labels on its vertices.
There is no internal symmetry of the form that the actions of $\symgp{k}$ are intended to break.
The actions of $\symgp{2}$ act on ordinary trees rooted at a \emph{directed} edge, with the nontrivial element $\tau \in \symgp{2}$ acting by reversing this orientation.
The resulting decomposition from the dissymmetry theorem in \cref{thm:dissymk} and the recursive functional equations of \cref{thm:funcdecompct} then clearly reduce to the classical dissymmetry analysis of ordinary trees.

\subsection{$2$-trees}
When $k=2$, there is a nontrivial symmetry at fronts (edges); two triangles may be joined at an edge in two distinct ways.
The imposition of a coherent orientation on a $2$-tree by directing one of its edges breaks this symmetry; the action of $\symgp{2}$ by reversal of these orientations gives unoriented $2$-trees as its orbits.
The defined action of $\symgp{3}$ on edge-rooted oriented triangles is simply the classical action of the dihedral group $D_{6}$ on a triangle, and its orbits are unoriented, unrooted triangles.
We further note that $\rho_{i}$ is the trivial map on $\symgp{2}$ and that $\rho_{i} \pbrac{\sigma} = \pmt{(12)}$ for $\sigma \in \symgp{3}$ if and only if $\sigma$ is an odd permutation, both regardless of $i$.
We then have that:
\begin{subequations}
  \label{eq:rest2trees}
  \begin{equation}
    \gci{\symgp{2}}{\cty{2}} = p_{1} \sbrac{y} \cdot \ci{\specname{E}} \circ \pbrac[\Big]{p_{1} \sbrac{x} \cdot \smashoperator{\prod_{c \in C \pbrac{\sigma}}} \gcieltvars{\symgp{2}}{\cty{2}}{e}{p_{\abs{c}} \sbrac{x}, p_{2 \abs{c}} \sbrac{x}, \dots; p_{\abs{c}} \sbrac{y}, p_{2 \abs{c}} \sbrac{y}, \dots}} \label{eq:ctyfuncci2}
  \end{equation}
  and
  \begin{equation}
    \gci{\symgp{3}}{\ctxy{2}} = p_{1} \sbrac{x} \cdot \smashoperator{\prod_{c \in C \pbrac{\sigma}}} \gci{\symgp{2}}{\cty{2}} \pbrac[\big]{\rho \pbrac{\sigma}^{\abs{c}}} \pbrac*{p_{\abs{c}} \sbrac{x}, p_{2 \abs{c}} \sbrac{x}, \dots; p_{\abs{c}} \sbrac{y}, p_{2 \abs{c}} \sbrac{y}, \dots}. \label{eq:ctxyfuncci2}
  \end{equation}
\end{subequations}
where, by abuse of notation, we let $\rho$ represent any $\rho_{i}$.
By the previous, the argument $\rho \pbrac{\sigma}^{\abs{c}}$ in \cref{eq:ctxyfuncci2} is $\tau$ if and only if $\sigma$ is an odd permutation and $c$ is of odd length.
This analysis and the resulting formulas for the cycle index $\ci{\kt{2}}$ are essentially equivalent to those derived in \cite{gessel:spec2trees}.

\appendix
\section{Enumerative tables}\label{c:enum}
With the recursive functional equations for cycle indices of \cref{s:ktcycind}, we can calculate the explicit cycle index for the species $\kt{k}$ to any finite degree we choose using computational methods; this cycle index can then be used to enumerate both unlabeled and labeled (at fronts, hedra, or both) $k$-trees up to a specified number $n$ of hedra (or, equivalently, $kn + 1$ fronts).
We have done so here for $k \leq 7$ and $n \leq 30$ using Sage 5.0 \cite{sage} using code available in \cref{c:code}.
The resulting values appear in \cref{tab:ktrees}.

We note that both unlabeled and hedron-labeled enumerations of $k$-trees stabilize in $k$:
\begin{theorem}
  \label{thm:ktreestab}
  For $k \geq n - 2$, the numbers of unlabeled and hedron-labeled $k$-trees with $n$ hedra are independent of $k$.
\end{theorem}
\begin{proof}
  We show that the species $\kt{k}$ and $\kt{k+1}$ have contact up to order $k+2$ by explicitly constructing a natural bijection.
  We note that in a $\pbrac{k+1}$-tree with no more than $k+2$ hedra, there will exist at least one vertex which is common to \emph{all} hedra.
  For any $k$-tree with no more than $k+2$ hedra, we can construct a $\pbrac{k+1}$-tree with the same number of hedra by adding a single vertex and connecting it by edges to every existing vertex; we can then pass labels up from the $\pbrac{k+1}$-cliques which are the hedra of the $k$-tree to the $\pbrac{k+2}$-cliques which now sit over them.
  The resulting graph will be a $\pbrac{k+1}$-tree whose $\pbrac{k+1}$-tree hedra are adjacent exactly when the $k$-tree hedra they came from were adjacent.
  Therefore, any two distinct $k$-trees will pass to distinct $\pbrac{k+1}$-trees.
  Similarly, for any $\pbrac{k+1}$-tree with no more than $k+2$ hedra, choose one of the vertices common to all the hedra and remove it, passing the labels of $\pbrac{k+1}$-tree hedra down to the $k$-tree hedra constructed from them; again, adjacency of hedra is preserved.
  This of course creates a $k$-tree, and for distinct $\pbrac{k+1}$-trees the resulting $k$-trees will be distinct.
  Moreover, by symmetry the result is independent of the choice of common vertex, in the case there is more than one.
\end{proof}
However, thus far we have neither determined a direct method for computing these stabilization numbers nor identified a straightforward combinatorial characterization of the structures they represent.

\clearpage
\begin{table}[htb]
  \centering
  \caption{Enumerative data for $k$-trees with $n$ hedra}
  \label{tab:ktrees}

  \subfloat[$k = 1$]{
    \label{tab:1trees}
    \begin{tabular}{r | r}
      $n$ & Unlabeled $1$-trees \\ \hline
      0 & 1 \\
      1 & 1 \\
      2 & 1 \\
      3 & 2 \\
      4 & 3 \\
      5 & 6 \\
      6 & 11 \\
      7 & 23 \\
      8 & 47 \\
      9 & 106 \\
      10 & 235 \\
      11 & 551 \\
      12 & 1301 \\
      13 & 3159 \\
      14 & 7741 \\
      15 & 19320 \\
      16 & 48629 \\
      17 & 123867 \\
      18 & 317955 \\
      19 & 823065 \\
      20 & 2144505 \\
      21 & 5623756 \\
      22 & 14828074 \\
      23 & 39299897 \\
      24 & 104636890 \\
      25 & 279793450 \\
      26 & 751065460 \\
      27 & 2023443032 \\
      28 & 5469566585 \\
      29 & 14830871802 \\
      30 & 40330829030
    \end{tabular}
  }
  \qquad
  \subfloat[$k = 2$]{
    \label{tab:2trees}
    \begin{tabular}{r | r}
      $n$ & Unlabeled $2$-trees \\ \hline
      0 & 1 \\
      1 & 1 \\
      2 & 1 \\
      3 & 2 \\
      4 & 5 \\
      5 & 12 \\
      6 & 39 \\
      7 & 136 \\
      8 & 529 \\
      9 & 2171 \\
      10 & 9368 \\
      11 & 41534 \\
      12 & 188942 \\
      13 & 874906 \\
      14 & 4115060 \\
      15 & 19602156 \\
      16 & 94419351 \\
      17 & 459183768 \\
      18 & 2252217207 \\
      19 & 11130545494 \\
      20 & 55382155396 \\
      21 & 277255622646 \\
      22 & 1395731021610 \\
      23 & 7061871805974 \\
      24 & 35896206800034 \\
      25 & 183241761631584 \\
      26 & 939081790240231 \\
      27 & 4830116366008952 \\
      28 & 24927175920361855 \\
      29 & 129047003236769110 \\
      30 & 670024248072778235
    \end{tabular}
  }
\end{table}

\begin{table}[htb]
  \centering
  \ContinuedFloat
  \caption*{Enumerative data for $k$-trees with $n$ hedra, continued}
  \subfloat[$k = 3$]{
    \label{tab:3trees}
    \begin{tabular}{r | r}
      $n$ & Unlabeled $3$-trees \\ \hline
      0 & 1 \\
      1 & 1 \\
      2 & 1 \\
      3 & 2 \\
      4 & 5 \\
      5 & 15 \\
      6 & 58 \\
      7 & 275 \\
      8 & 1505 \\
      9 & 9003 \\
      10 & 56931 \\
      11 & 372973 \\
      12 & 2506312 \\
      13 & 17165954 \\
      14 & 119398333 \\
      15 & 841244274 \\
      16 & 5993093551 \\
      17 & 43109340222 \\
      18 & 312747109787 \\
      19 & 2286190318744 \\
      20 & 16826338257708 \\
      21 & 124605344758149 \\
      22 & 927910207739261 \\
      23 & 6945172081954449 \\
      24 & 52225283886702922 \\
      25 & 394398440097305861 \\
      26 & 2990207055800156659 \\
      27 & 22753619938517594709 \\
      28 & 173727411594289881739 \\
      29 & 1330614569159767263501 \\
      30 & 10221394007530945428347
    \end{tabular}
  }
  \qquad
  \subfloat[$k = 4$]{
    \label{tab:4trees}
    \begin{tabular}{r | r}
      $n$ & Unlabeled $4$-trees \\ \hline
      0 & 1 \\
      1 & 1 \\
      2 & 1 \\
      3 & 2 \\
      4 & 5 \\
      5 & 15 \\
      6 & 64 \\
      7 & 331 \\
      8 & 2150 \\
      9 & 15817 \\
      10 & 127194 \\
      11 & 1077639 \\
      12 & 9466983 \\
      13 & 85252938 \\
      14 & 782238933 \\
      15 & 7283470324 \\
      16 & 68639621442 \\
      17 & 653492361220 \\
      18 & 6276834750665 \\
      19 & 60759388837299 \\
      20 & 592227182125701 \\
      21 & 5808446697002391 \\
      22 & 57289008242377068 \\
      23 & 567939935463185078 \\
      24 & 5656700148512008902 \\
      25 & 56583199285317631541 \\
      26 & 568236762643725657852 \\
      27 & 5727423267612393252616 \\
      28 & 57924486783495226147615 \\
      29 & 587672090447840337304025 \\
      30 & 5979782184127687211698807
    \end{tabular}
  }
\end{table}

\begin{table}[htb]
  \centering
  \ContinuedFloat
  \caption*{Enumerative data for $k$-trees with $n$ hedra, continued}
  \subfloat[$k = 5$]{
    \label{tab:5trees}
    \begin{tabular}{r | r}
      $n$ & Unlabeled $5$-trees \\ \hline
      0 & 1 \\
      1 & 1 \\
      2 & 1 \\
      3 & 2 \\
      4 & 5 \\
      5 & 15 \\
      6 & 64 \\
      7 & 342 \\
      8 & 2321 \\
      9 & 18578 \\
      10 & 168287 \\
      11 & 1656209 \\
      12 & 17288336 \\
      13 & 188006362 \\
      14 & 2105867058 \\
      15 & 24108331027 \\
      16 & 280638347609 \\
      17 & 3310098377912 \\
      18 & 39462525169310 \\
      19 & 474697793413215 \\
      20 & 5754095507495584 \\
      21 & 70216415130786725 \\
      22 & 861924378411516159 \\
      23 & 10636562125193377459 \\
      24 & 131890971196221692874 \\
      25 & 1642577274341274449247 \\
      26 & 20538830517384955820622 \\
      27 & 257767439475728146293796 \\
      28 & 3246108646710813383678978 \\
      29 & 41008581189552637540038747 \\
      30 & 519599497193547405843864376
    \end{tabular}
  }
  \quad
  \subfloat[$k = 6$]{
    \label{tab:6trees}
    \begin{tabular}{r | r}
      $n$ & Unlabeled $6$-trees \\ \hline
      0 & 1 \\
      1 & 1 \\
      2 & 1 \\
      3 & 2 \\
      4 & 5 \\
      5 & 15 \\
      6 & 64 \\
      7 & 342 \\
      8 & 2344 \\
      9 & 19090 \\
      10 & 179562 \\
      11 & 1878277 \\
      12 & 21365403 \\
      13 & 258965451 \\
      14 & 3294561195 \\
      15 & 43472906719 \\
      16 & 589744428065 \\
      17 & 8171396893523 \\
      18 & 115094557122380 \\
      19 & 1642269376265063 \\
      20 & 23679803216530017 \\
      21 & 344396036645439675 \\
      22 & 5045351124912000756 \\
      23 & 74375422235109338507 \\
      24 & 1102368908826371717478 \\
      25 & 16417712341047912048640 \\
      26 & 245566461812077209025580 \\
      27 & 3687384661929075391318298 \\
      28 & 55566472746158319169779382 \\
      29 & 840092106663809502446963972 \\
      30 & 12739517442131428048314937036
    \end{tabular}
  }
\end{table}

\begin{table}[htb]
  \centering
  \ContinuedFloat
  \caption*{Enumerative data for $k$-trees with $n$ hedra, continued}
  \subfloat[$k = 7$]{
    \label{tab:7trees}
    \begin{tabular}{r | r}
      $n$ & Unlabeled $7$-trees \\ \hline
      0 & 1 \\
      1 & 1 \\
      2 & 1 \\
      3 & 2 \\
      4 & 5 \\
      5 & 15 \\
      6 & 64 \\
      7 & 342 \\
      8 & 2344 \\
      9 & 19137 \\
      10 & 181098 \\
      11 & 1922215 \\
      12 & 22472875 \\
      13 & 284556458 \\
      14 & 3849828695 \\
      15 & 54974808527 \\
      16 & 819865209740 \\
      17 & 12655913153775 \\
      18 & 200748351368185 \\
      19 & 3253193955012557 \\
      20 & 53619437319817482 \\
      21 & 895778170144927928 \\
      22 & 15129118461773051724 \\
      23 & 257812223121779545108 \\
      24 & 4426056869082751747930 \\
      25 & 76463433541541506345648 \\
      26 & 1328088941166844504424628 \\
      27 & 23175796698013212039339479 \\
      28 & 406103563562864890670029228 \\
      29 & 7142350290468621849814034057 \\
      30 & 126034923903699365819345698783
    \end{tabular}   
  }
\end{table}

\begin{table}[htb]
  \centering
  \ContinuedFloat
  \caption*{Enumerative data for $k$-trees with $n$ hedra, continued}
  \subfloat[$k = 8$]{
    \label{tab:8trees}
    \begin{tabular}{r | r}
      $n$ & Unlabeled $8$-trees \\ \hline
      0 & 1 \\
      1 & 1 \\
      2 & 1 \\
      3 & 2 \\
      4 & 5 \\
      5 & 15 \\
      6 & 64 \\
      7 & 342 \\
      8 & 2344 \\
      9 & 19137 \\
      10 & 181204 \\
      11 & 1926782 \\
      12 & 22638677 \\
      13 & 289742922 \\
      14 & 3996857019 \\
      15 & 58854922207 \\
      16 & 916955507587 \\
      17 & 14988769972628 \\
      18 & 255067524402905 \\
      19 & 4487202163529135 \\
      20 & 81112295567987808 \\
      21 & 1498874117898285574 \\
      22 & 28195965395340358096 \\
      23 & 538126404726276758908 \\
      24 & 10391826059632904271057 \\
      25 & 202624626664206041379718 \\
      26 & 3982593421723767068438772 \\
      27 & 78804180647706388187446055 \\
      28 & 1568191570016583843925943321 \\
      29 & 31359266621157738864915907470 \\
      30 & 629755261439815181073415721542
    \end{tabular}
  }
\end{table}

\begin{table}[htb]
  \centering
  \ContinuedFloat
  \caption*{Enumerative data for $k$-trees with $n$ hedra, continued}
  \subfloat[$k = 9$]{
    \label{tab:9trees}
    \begin{tabular}{r | r}
      $n$ & Unlabeled $9$-trees \\ \hline
      0 & 1 \\
      1 & 1 \\
      2 & 1 \\
      3 & 2 \\
      4 & 5 \\
      5 & 15 \\
      6 & 64 \\
      7 & 342 \\
      8 & 2344 \\
      9 & 19137 \\
      10 & 181204 \\
      11 & 1927017 \\
      12 & 22652254 \\
      13 & 290351000 \\
      14 & 4019973352 \\
      15 & 59642496465 \\
      16 & 941751344429 \\
      17 & 15724551551655 \\
      18 & 275926445572426 \\
      19 & 5057692869843759 \\
      20 & 96275031338911591 \\
      21 & 1892687812366295682 \\
      22 & 38234411627616084843 \\
      23 & 790120238796588845615 \\
      24 & 16638524087850961727575 \\
      25 & 355878246778832856290372 \\
      26 & 7710423952280397990026132 \\
      27 & 168843592748278228259801752 \\
      28 & 3730285520855433827693340329 \\
      29 & 83027821492843727307516904184 \\
      30 & 1859625249087075723295908757282
    \end{tabular}
  }
\end{table}

\begin{table}[htb]
  \centering
  \ContinuedFloat
  \caption*{Enumerative data for $k$-trees with $n$ hedra, continued}
  \subfloat[$k = 10$]{
    \label{tab:10trees}
    \begin{tabular}{r | r}
      $n$ & Unlabeled $10$-trees \\ \hline
      0 & 1 \\
      1 & 1 \\
      2 & 1 \\
      3 & 2 \\
      4 & 5 \\
      5 & 15 \\
      6 & 64 \\
      7 & 342 \\
      8 & 2344 \\
      9 & 19137 \\
      10 & 181204 \\
      11 & 1927017 \\
      12 & 22652805 \\
      13 & 290391147 \\
      14 & 4022154893 \\
      15 & 59741455314 \\
      16 & 945737514583 \\
      17 & 15871943695637 \\
      18 & 281035862707569 \\
      19 & 5226147900656616 \\
      20 & 101612006684523937 \\
      21 & 2056425123910104429 \\
      22 & 43127730369661586804 \\
      23 & 933229734601789336024 \\
      24 & 20749443766669472108394 \\
      25 & 472211306357077710523863 \\
      26 & 10961384502758318928846970 \\
      27 & 258737420965101611169934566 \\
      28 & 6193917223279376307682721853 \\
      29 & 150039339181032274342778699887 \\
      30 & 3670778410024403632885217999313
    \end{tabular}
  }
\end{table}

\clearpage
\section{Code listing}\label{c:code}
The recursive functional equations in \cref{eq:ctyogf,eq:ctxyogf,eq:ktgf} characterize the ordinary generating function $\tilde{\mathfrak{a}}_{k} \pbrac{x}$ for unlabeled general $k$-trees.
Code to compute the coefficients of this generating function using the computer algebra system Sage 5.0 \cite{sage} explicitly follows in \cref{lst:ktcode}.
Specifically, the generating function for unlabeled $k$-trees may be computed to degree $n$ by copying the included code into a Sage notebook, modifying the final line with the desired values of $k$ and $n$, and executing.

This code takes full advantage of \cref{lem:rhoprod,lem:ciclassfunc} to minimize the number of distinct calculations which must be performed; as a result, it is able to compute the number of $k$-trees on up to $n$ hedra quickly even for relatively large $k$ and $n$.
For example, the first thirty terms of the generating function for $8$-trees in \cref{tab:8trees} were computed on a modern desktop-class computer in approximately two minutes.

\begin{lstlisting}[caption=Sage code to compute numbers of $k$-trees), label=lst:ktcode, language=Python]
# Set up a ring of formal power series
psr = PowerSeriesRing(QQ, 'x')
x = psr.gen()

# Compute the generating function for unlabeled Y-rooted k-trees fixed by permutations of a given cycle type mu.
# Note that mu should partition k
@cached_function
def unlY(mu, n):
    if n <= 0:
        return psr(1)
    else:
        ystretcher = lambda c, part: unlY(Partition(partition_power(part, c)), floor((n-1)/c)).subs({x:x**c})
        descendant_pseries = lambda part: prod(ystretcher(c, part) for c in part)
        return sum(x**i/i * descendant_pseries(partition_power(mu, i)).subs({x:x**i}) for i in xrange(1, n+1)).exp(n+1)

# Compute the generating function for unlabeled XY-rooted k-trees fixed by permutations of a given cycle type mu.
# Note that mu should partition k+1
@cached_function
def unlXY(mu, n):
    if n <= 0:
        return psr(0)
    else:
        ystretcher = lambda c: unlY(Partition(partition_power(mu, c)[:-1]), floor((n-1)/c)).subs({x:x**c})
        return (x * prod(ystretcher(c) for c in mu)).add_bigoh(n+1)

# Compute the generating functions for unlabeled X-, Y-, and XY-rooted k-trees using quotients
ax = lambda k, n: sum(1/mu.aut() * unlXY(mu, n) for mu in Partitions(k+1))
ay = lambda k, n: sum(1/mu.aut() * unlY(mu, n) for mu in Partitions(k))
axy = lambda k, n: sum(1/mu.aut() * unlXY(Partition(mu + [1]), n) for mu in Partitions(k))

# Compute the generating function for unlabeled un-rooted k-trees using the dissymmetry theorem
a = lambda k, n: ax(k, n) + ay(k, n) - axy(k, n)

# Print the result
# ALERT: User must substitue values for k and n (number of hedra)
print a(kval, nval)
\end{lstlisting}

\clearpage

\section*{Acknowledgments}
The author wishes to express his profound gratitude to Ira Gessel, who served as advisor and committee chair for the dissertation in which this research originated.
His guidance, advice, and insightful mentorship were invaluable throughout that process.

The author also offers his thanks to an anonymous referee, whose close reading and thorough commentary led to the correction of numerous small errors and oversights and whose broader recommendations led to significant improvement of the structure and narrative flow of the paper as a whole.

\bibliographystyle{plain} %bibliography
\bibliography{sources}

\end{document}